\documentclass{statsoc}

\usepackage[a4paper]{geometry}

\usepackage[utf8x]{inputenc}
\usepackage[english]{babel}
\usepackage{amsfonts}
\usepackage{amssymb}
\usepackage{graphicx}
\usepackage{mathrsfs}
\usepackage{bm}
\usepackage{bbm}

\usepackage{amsthm}
\usepackage{dsfont}
\usepackage{color}

\usepackage{mathtools}
\usepackage{mathrsfs}
\usepackage{algorithmic}
\usepackage{algorithm}
\usepackage{multirow}
\usepackage{enumerate}
\usepackage{url}

\usepackage{lmodern}
\usepackage{anyfontsize}

\usepackage[table]{xcolor}
\usepackage{picture}

\usepackage{array}
\usepackage{multirow}
\usepackage{enumerate}
\usepackage{booktabs}

\usepackage{soul}

\numberwithin{equation}{section}

\theoremstyle{plain}\newtheorem{thm}{Theorem}[section]
\theoremstyle{plain}\newtheorem{corollary}[thm]{Corollary}
\theoremstyle{definition}\newtheorem{remark}[thm]{Remark}
\theoremstyle{plain}
\theoremstyle{plain}
\theoremstyle{plain}\newtheorem{proposition}[thm]{Proposition}
\theoremstyle{definition}

\newcommand{\ud}{\mbox{d}}
\newcommand{\eps}{\varepsilon}

\DeclareMathOperator{\var}{\mathsf{Var}}

\DeclareMathOperator{\E}{\mathsf{E}}

\newcommand{\prob}{\mathsf{P}}

\newcolumntype{x}[1]{%
	>{\raggedleft\hspace{0pt}}p{#1}}%

\DeclareMathAlphabet{\mathitbf}{OML}{cmm}{b}{it}

\definecolor{seda}{gray}{0.9}

\makeatletter 
\newcommand{\vast}{\bBigg@{4}}
\newcommand{\Vast}{\bBigg@{5}}
\makeatother

\floatname{algorithm}{Procedure}

\newcommand{\bA}{\bm{A}}
\newcommand{\ba}{\bm{a}}
\newcommand{\bB}{\bm{B}}

\newcommand{\bc}{\bm{c}}
\newcommand{\bb}{\bm{b}}
\newcommand{\bY}{\bm{Y}}
\newcommand{\bZ}{\bm{Z}}
\newcommand{\bX}{\bm{X}}

\newcommand{\bU}{\bm{U}}
\newcommand{\bV}{\bm{V}}
\newcommand{\bW}{\bm{W}}
\newcommand{\bR}{\bm{R}}

\newcommand{\bI}{\bm{I}}

\newcommand{\bM}{\bm{M}}

\newcommand{\bu}{\bm{u}}
\newcommand{\bv}{\bm{v}}

\newcommand{\bD}{\bm{D}}
\newcommand{\zero}{\bm{0}}

\newcommand{\balpha}{\bm{\alpha}}
\newcommand{\bbeta}{\bm{\beta}}
\newcommand{\bgamma}{\bm{\gamma}}
\newcommand{\beps}{\bm{\varepsilon}}
\newcommand{\bvartheta}{\bm{\vartheta}}

\newcommand{\btheta}{\bm{\Theta}}

\newcommand{\bdelta}{\bm{\delta}}
\newcommand{\bDelta}{\bm{\Delta}}
\newcommand{\bSigma}{\bm{\Sigma}}
\newcommand{\bGamma}{\bm{\Gamma}}

\newcommand{\bvarphi}{\bm{\varphi}}

\newcommand{\Oo}{\mathcal{O}}

\renewcommand{\P}{\mathsf{P}} 

\newcommand{\dist}{\mathsf{D}} 
\newcommand{\weak}{\mathsf{D}[0,1]}

\newcommand{\Op}{\ensuremath{O_{\prob}}} 
\newcommand{\op}{\ensuremath{o_{\prob}}} 

\newtheorem{assumpC}{Assumption}

\newtheorem{assumpD}{Assumption}

\newtheorem{assumpE}{Assumption}

\newtheorem{assumpV}{Assumption}

\usepackage{etoolbox}

\makeatletter
\patchcmd{\@makecaption}
{\parbox}
{\advance\@tempdima-\fontdimen2} 
{}{}
\makeatother 

\allowdisplaybreaks

\usepackage{natbib}
\title[Changepoint in Linear Relations]{Changepoint in Linear Relations}
\author[M.~Pe\v{s}ta]{Michal Pe\v{s}ta} \address{Charles University, Prague, Czech Republic.}\email{Michal.Pesta@mff.cuni.cz}

\begin{document}
	
\begin{abstract}
Linear relations, containing measurement errors in input and output data, are considered. Parameters of these so-called errors-in-variables models can change at some unknown moment. The aim is to test whether such an~unknown change has occurred or not. For instance, detecting a~change in trend for a~randomly spaced time series is a~special case of the investigated framework. The designed changepoint tests are shown to be consistent and involve neither nuisance parameters nor tuning constants, which makes the testing procedures effortlessly applicable. A~changepoint estimator is also introduced and its consistency is proved. A~boundary issue is avoided, meaning that the changepoint can be detected when being close to the extremities of the observation regime. As a~theoretical basis for the developed methods, a~weak invariance principle for the smallest singular value of the data matrix is provided, assuming weakly dependent and non-stationary errors. The results are presented in a~simulation study, which demonstrates computational efficiency of the techniques. The completely data-driven tests are illustrated through a~calibration problem, however, the methodology can be applied to other areas such as clinical measurements, dietary assessment, computational psychometrics, or environmental toxicology as manifested in the paper.
\end{abstract}


\section{Introduction and main aims}
If measured input and output data are supposed to be in some linear relations, then it is of particular interest to detect whether impact of the input characteristics has changed over time on the output observables. Moreover, only \emph{error-prone surrogates} of the unobservable input-output characteristics are in hand instead of a~precise measurement. Despite the fact that the relations and, consequently, suitable underlying stochastic models are linearly defined, the possible estimates and the corresponding inference may be highly non-linear~\citep{gleser}. It becomes even more challenging to handle measurement errors in input and output data simultaneously, when the linear relations are subject to change at some unknown time point---\emph{changepoint}.

There is a~vast literature aimed at linear relations modeled through so-called \emph{measurement error models} or \emph{errors-in-variables models} (for an~over\-view, see \cite{Fuller1987}, \cite{huffel}, \cite{CRSC2006}, \cite{Buonaccorsi2010}, or \cite{Yi2017}), but very little has been explored in the changepoint analysis for these models yet. A~change in regression has been explored thoroughly, cf.~\cite{H1995} or~\cite{AHHK2008}. However, such a~framework does not cover the case of measurement error models. Maximum likelihood approach~\citep{ChH1997,SS2002} and Bayesian approach~\citep{CRW1999,GK2001} to the changepoint estimation in the measurement error models were applied, both requiring parametric distributional assumptions on the errors. \cite{KMV2007} estimated the changepoint in the input data only. A~change in the variance parameter of the normally distributed errors within the measurement error models was investigated by~\cite{DTJM2016}. All of these mentioned contributions dealt with the changepoint estimation solely. Our \emph{main goal} is to test for a~possible change in the parameters relating the input and output data, both encumbered by some errors. Consequently, if a~change is detected, we aim to estimate it. By our best knowledge, we are not aware of any similar results even for the independent and identically distributed errors. Additionally to that, our changepoint tests are supposed to be \emph{nuisance-parameter-free}, \emph{distributional-free}, and to allow for very general error structures.

\subsection{Outline}
The paper is organized as follows: In the next section, our data model for the changepoint in errors-in-variables is introduced and several practical motivations for such a~model are given. Section~\ref{sec:SWIP} contains a~spectral weak invariance principle for weakly dependent and non-stationary random variables. It serves as the main theoretical tool for the consequent inference. The technical assumptions are discussed as well. Two test statistics for the changepoint detection are proposed in Section~\ref{sec:test}. Consequently, their asymptotic behavior is derived under the null as well as under the alternative hypothesis. Moreover, a~consistent changepoint estimator is introduced. Section~\ref{sec:simulations} contains a~simulation study that compares finite sample performance of the investigated tests. It numerically emphasizes the advantages of the proposed detection procedures. A~practical application of the developed approach to a~calibration problem is presented in Subsection~\ref{subsec:Papplication}. On the other hand, a~theoretical application to randomly spaced time series is performed in Subsection~\ref{subsec:Tapplication}. Afterwards, our conclusion follows. Proofs are given in the Appendix~\ref{sec:proofs}.

\section{Changepoint in errors-in-variables}\label{sec:CPTinEIV}
\emph{Errors-in-variables} (EIV) or also called \emph{measurement error} model
\begin{equation}\label{eq:M}\tag{$\mathcal{M}$}
\bX=\bZ+\btheta
\end{equation}
and
\begin{equation}\label{eq:H0}\tag{$\mathcal{H}_0$}
\bY=\bZ\bbeta+\beps
\end{equation}
is considered, where $\bbeta\in\mathbb{R}^p$ is a~vector of unknown \emph{regression parameters} possibly subject to change, $\bX\in\mathbb{R}^{n\times p}$ and $\bY\in\mathbb{R}^{n\times 1}$ consist of \emph{observable random} variables ($\bX$ are covariates and $\bY$ is a~response), $\bZ\in\mathbb{R}^{n\times p}$ consists of \emph{unknown constants} and has full rank, $\beps\in\mathbb{R}^{n\times 1}$ and $\btheta\in\mathbb{R}^{n\times p}$ are \emph{random errors}. This setup can be extended to a~\emph{multivariate case}, where $\bbeta\in\mathbb{R}^{p\times q}$, $\bY\in\mathbb{R}^{n\times q}$, and $\beps\in\mathbb{R}^{n\times q}$, $q\geq 1$, see Subsection~\ref{subsec:extension}.

The EIV model~\eqref{eq:M}--\eqref{eq:H0} with non-random unknown constants $\bZ$ is sometimes called \emph{functional EIV} model~\citep{Fuller1987,hall}. On the other hand, a~different approach may handle $\bZ$ as random covariates, which is called \emph{structural EIV} model~\citep{ChH1997}. \cite{S2000} stated: `However, functional models played an~important role in the study of measurement error models and in statistics more generally.' And here, we will concentrate on the functional EIV model not because of this matter-of-fact quote, but because we wish to demonstrate a~distributional-free approach, where `no, or only minimal, assumptions are made about the distribution of the $\bX$s' \citep{CRSC2006}, as challenged in the introduction. Nevertheless with respect to derivation of the forthcoming theory for the functional EIV model, changing some technical assumptions would allow to prove suitable results for the structural case as well.

To estimate the unknown parameter~$\bbeta$, one usually minimizes the \emph{Frobenius matrix norm} of the errors $[\btheta,\beps]$, see~\cite{golub}. This approach leads to a~\emph{total least squares} (TLS) estimate $\hat{\bbeta}=(\bX^{\top}\bX-\lambda_{min}([\bX,\bY]^{\top}[\bX,\bY])\bI_p)^{-1}\bX^{\top}\bY$, where $\lambda_{min}(\bM)$ is the smallest eigenvalue of the matrix~$\bM$ and $\bI_p$ is a~$(p\times p)$ identity matrix. Geometrically speaking, the Frobenius norm tries to minimize the \emph{orthogonal distance} between the observations and the fitted hyperplane. Therefore, the TLS are usually known as \emph{orthogonal regression}. One can generalize this method by replacing the Frobenius norm by \emph{any unitary invariance matrix norm}, which surprisingly yields the same TLS estimate, having interesting invariance and equivariance properties~\citep{P2016}. The TLS estimate is shown to be strongly and weakly consistent~\citep{gleser,gallo,P2011} as well as to be asymptotically normal~\citep{gallophd,P2013,P2017} under various conditions.


We aim to detect a~possible change in the linear relation parameter~$\bbeta$. The interest lies in testing the \emph{null hypothesis}~\eqref{eq:H0} of all observations~$Y_i$'s being random variables having expectations~$\bZ_{i,\bullet}\bbeta$'s. Our goal is to test against the alternative of the first $\tau$ outcome observations have expectations~$\bZ_{i,\bullet}\bbeta$'s and the remaining $n-\tau$ observations come from distributions with expectations $\bZ_{i,\bullet}(\bbeta+\bdelta)$'s, where $\bdelta\neq\zero$. A~`row-column' notation for a~matrix $\bM$ is used in this manner: $\bM_{i,\bullet}$ denotes the $i$th row of~$\bM$ and $\bM_{\bullet,j}$ corresponds to the $j$th column of~$\bM$. Furthermore, if $i\in\mathbb{N}_0$, then $\bM_{i}$ stays for the first~$i$ rows of~$\bM$ and $\bM_{-i}$ represents the remaining $n-i$ rows of~$\bM$, when the first~$i$ rows are deleted. Now more precisely, our \emph{alternative hypothesis} is
\begin{equation}\label{eq:HA}\tag{$\mathcal{H}_A$}
\bY_{\tau}=\bZ_{\tau}\bbeta+\beps_{\tau}\quad\mbox{and}\quad\bY_{-\tau}=\bZ_{-\tau}(\bbeta+\bdelta)+\beps_{-\tau}.
\end{equation}
Here, $\bdelta\equiv\bdelta(n)\neq\zero$ is an~unknown vector parameter representing the size of change and is possibly depending on~$n$. The \emph{changepoint} $\tau\equiv\tau(n)<n$ is also an~unknown scalar parameter, which depends on~$n$ as well. Although, $\bbeta$ is considered to be independent of~$n$. One may also think of the changepoint in errors-in-variables framework as \emph{segmented regression with measurement errors}, cf.~\cite{SS2002}.

\subsection{Intercept and fixed regressors}
Note that the EIV model~\eqref{eq:M}--\eqref{eq:H0} has no intercept and all the covariates are encumbered by some errors. To overcome such a~restriction, one can think of an~extended regression model, where some explanatory variables are \emph{subject to error} and some are measured \emph{precisely}. I.e., $\bY=\bW\bgamma+\bZ\bbeta+\beps$, where $\bW$ are \emph{observable true} and $\bZ$ are \emph{unobservable true} constants, both having full rank. Regression parameters $\bgamma$ and $\bbeta$ remain unknown. Then, the non-random (fixed) \emph{intercept} can be incorporated into the regression model by setting one column of the matrix~$\bW$ equal to $[1,\ldots,1]^{\top}$. Consequently, we may \emph{project out} exact observations using projection matrix $\bR:=\bI_n-\bW(\bW^{\top}\bW)^{-1}\bW^{\top}$. Notice that $\bR$ is symmetric and idempotent. Finally, one may work with $\bR\bY=\bR\bZ\bbeta+\bR\beps$ instead of~\eqref{eq:H0}.

\subsection{Motivations}
The proposed class of models---\emph{errors-in-variables with changepoint}---is very rich and general. Our approach and results are motivated in the context of several applications taken from chemistry, biological sciences, medicine, and epidemiological studies.

\vspace{-0.3cm}
\paragraph{Case~1: Assessing agreement in clinical measurement.}
Direct measurement of cardiac stroke volume or blood pressure without adverse effects is difficult or even impossible. The true values remain unknown. Indirect methods are, therefore, used instead. When a~new measurement technique is developed, it has to be evaluated by comparison with an~established technique rather than with the true quantity~\citep{BA1986}. Clinicians need to test whether both measurement techniques agree sufficiently. Thereafter, the old technique may be replaced by the new one.

\vspace{-0.3cm}
\paragraph{Case~2: Nutritional epidemiology.}
\cite{SS2002} analyzed data from a~nutritional study that investigates the relation between dietary folate intake (calories adjusted $\mu$g/day) on plasma homocysteine concentration ($\mu$mol/liter of blood). There exists a~suspicion that serum homocysteine is significantly elevated when ingested folate is below a~certain changepoint. Moreover, the analysis used estimates of folate that were developed with a~food frequency questionnaire, which is recognized to be imperfect.

\vspace{-0.3cm}
\paragraph{Case~3: Psychometric testing.}
Let us think of two psychometric instruments: unspeeded 15-item vocabulary tests and highly speeded 75-item vocabulary tests, cf.~\cite{L1973}. The results of both tests are error-prone. Within a~group of people, there is a~speculation that individuals with an~unspeeded test's result exceeding some unknown level should perform dramatically better in the highly speeded test.

\vspace{-0.3cm}
\paragraph{Case~4: Environmental toxicology.}
A~threshold limiting value in toxicology is the dose of a~toxin or a~substance under which there is harmless or insignificant influence on some response. In a~dose-response relationship, both of them are measured with errors. And the goal is to set the threshold limiting value. Such a~problem was dealt by~\cite{GK2001} using fully Bayesian approach. Moreover, a~similar task regarding the NO$_2$ concentration is discussed by~\cite{S2000}.

\vspace{-0.3cm}
\paragraph{Case~5: Device calibration.}
Later on in Subsection~\ref{subsec:Papplication}, we concentrate in more details on the \emph{calibration} task and \emph{exemplify the proposed methodology} through analysis of data from a~calibrated device and a~casual device (needs to be calibrated) in order to demonstrate practical efficiency of our detection method.

\bigskip

Besides that, there are many other applications of the changepoint within the linear relations framework in, for instance, glaciology~\citep{watson}, empirical economics~\citep{ChH1997}, dietary assessment~\citep{CRW1999}, or image forensics~\citep{RL2014}.

\section{Spectral weak invariance principle}\label{sec:SWIP}
A~theoretical device is going to be developed in order to construct the changepoint tests. The \emph{smallest eigenvalue} of~$\bSigma^{-1}[\bX,\bY]^{\top}[\bX,\bY]$---the squared smallest singular value of~$[\bX,\bY]\bSigma^{-1/2}$, i.e., the data matrix~$[\bX,\bY]$ multiplied by the inverse of a~matrix square root from the error variance structure (cf.~subsequent Assumption~\ref{assump:EIVerrors})---plays a~key role. We proceed to the assumptions that are needed for deriving forthcoming asymptotic results. Henceforth, $\xrightarrow{\prob}$ denotes convergence in probability, $\xrightarrow{\dist}$ convergence in distribution, $\xrightarrow[n\to\infty]{\weak}$ weak convergence in the Skorokhod topology $\weak$ of c\`adl\`ag functions on $[0,1]$, and $[x]$ denotes the integer part of the real number~$x$.

\subsection{Assumptions}\label{subsec:assump}
Firstly, a~\emph{design assumption} on the unobservable regressors is needed.

\begin{assumpD}\label{assump:EIVdesign}
$\bDelta_{t}:=\lim_{n\to\infty}n^{-1}\bZ_{[nt]}^{\top}\bZ_{[nt]}$, $\bDelta_{-t}:=\lim_{n\to\infty}n^{-1}\bZ_{-[nt]}^{\top}\bZ_{-[nt]}$ for every $t\in(0,1)$, and $\bDelta:=\lim_{n\to\infty}n^{-1}\bZ^{\top}\bZ$ are positive definite.
\end{assumpD}

It basically says that the error-free design points do not concentrate to close to each other (i.e., strict positive definiteness) and, simultaneously, they do not spread-out too far (i.e., existence of limits). For example in one-dimensional case (i.e., $p=1$), a~simple equidistant design, where $Z_{i,1}=i/(n+1)$, provides $\bDelta_{t}=t^3/3$ and $\bDelta=1/3$.

Prior to postulating an~\emph{errors' assumption}, we summarize the notion of \emph{strong mixing} ($\alpha$-mixing) dependence in more detail, which will be imposed on the model's errors. Suppose that $\{\xi_n\}_{n=1}^{\infty}$ is a~sequence of random elements on a~probability space $(\Omega,\mathcal{F},\P)$. For sub-$\sigma$-fields $\mathcal{A},\mathcal{B}\subseteq\mathcal{F}$, let $\alpha(\mathcal{A}|\mathcal{B}):=\sup_{A\in\mathcal{A},B\in\mathcal{B}}\left|\P(A\cap B)-\P(A)\P(B)\right|$. Intuitively, $\alpha(\cdot|\cdot)$ measures the dependence of the events in $\mathcal{B}$ on those in $\mathcal{A}$. 
There are many ways in which one can describe weak dependence or, in other words, \emph{asymptotic independence} of random variables, see~\cite{Bradley2005}. Considering a~filtration $\mathcal{F}_m^n=\sigma\{\xi_i\in\mathcal{F},m\leq i\leq n\}$, sequence $\{\xi_n\}_{n=1}^{\infty}$ of random variables is said to be \emph{strong mixing} ($\alpha$-mixing) if $\alpha(\xi_{\circ},n)=\sup_{k\in\mathbb{N}}\alpha(\mathcal{F}_{1}^k|\mathcal{F}_{k+n}^{\infty})\to 0$ as $n\to\infty$. \cite{Anderson1958} comprehensively analyzed a~class of $m$-dependent processes. They are $\alpha$-mixing, since they are finite order ARMA processes with innovations satisfying \emph{Doeblin's condition} \citep[p.~168]{Billingsley1968}. Finite order processes, which do not satisfy Doeblin's condition, can be shown to be $\alpha$-mixing \citep[pp.~312--313]{IL1971}. \cite{Rosenblatt1971} provides general conditions under which stationary Markov processes are $\alpha$-mixing. Since functions of mixing processes are themselves mixing \citep{Bradley2005}, time-varying functions of any of the processes just mentioned are mixing as well. This means that the class of the $\alpha$-mixing processes is sufficiently large for the further practical applications and that is why we chose such a~mixing condition. 

\begin{assumpE}\label{assump:EIVerrors}
$\{[\btheta_{n,\bullet},\eps_n]\}_{n=1}^{\infty}$ is a~sequence of $\alpha$-mixing absolutely continuous random vectors having zero mean and a~variance matrix~$\sigma^2\bSigma$ with an~unknown $\sigma^2>0$ and a~known positive definite~$\bSigma=\begin{bmatrix}
\bSigma_{\btheta} & \bSigma_{\btheta,\beps}\\
\bSigma_{\btheta,\beps}^{\top} & 1
\end{bmatrix}$ such that $\alpha([\btheta_{\circ,\bullet},\eps_{\circ}],n)=\Oo(n^{-1-\varpi})$ as $n\to\infty$ for some $\varpi>0$, $\sup_{n\in\mathbb{N}} Z_{n,j}^2<\infty$, $\sup_{n\in\mathbb{N}}\E|\Theta_{n,j}|^{4+\omega}<\infty$, $j\in\{1,\ldots,p\}$, and $\sup_{n\in\mathbb{N}}\E|\eps_n|^{4+\omega}<\infty$ for some $\omega>0$ such that $\omega\varpi>2$.
\end{assumpE}

Let us emphasize that the sequence of the errors \emph{do not have to be stationary}. The assumption of an~unknown~$\sigma^2$ and a~known~$\bSigma$ implies that we know the ratio of any pair of covariances in advance. In the simplest situation, a~homoscedastic covariance structure of the within-individual errors $[\btheta_{n,\bullet},\eps_n]$ can be assumed (i.e., $\bSigma=\bI_{p+1}$), if prior experience or essence of the analyzed problem allow for that. On the other hand, if the covariance matrix $\bSigma$ is unknown, it can be estimated when possessing \emph{replicate measurements} or \emph{validation data} as commented by~\cite{S2000}. There are various approaches proposed to serve this purpose. In order ot mention at least some of them, we refer to \cite{ChR2006}, \cite{GL2011}, \cite{P2013}, or \cite{LML2019}. On the top of that, we have to bear in mind that~$\bSigma$ cannot be completely unspecified. \cite{N1977} showed that if~$\bSigma$ is unrestricted, no strongly consistent estimator for~$\bbeta$ can exist even under normally distributed errors.


Furthermore, a~\emph{variance assumption} for the misfit disturbances is stated. It can be considered as a~typical assumption for the \emph{long-run variance} of residuals. Let us denote $\bSigma^{-1/2}=\begin{bmatrix}
\bar{\bSigma}_{\btheta} & \bar{\bSigma}_{\btheta,\beps}\\
\bar{\bSigma}_{\btheta,\beps}^{\top} & \bar{\Sigma}_{\beps}
\end{bmatrix}$ a~symmetric square root of~$\bSigma^{-1}$, where $\bar{\Sigma}_{\beps}\in\mathbb{R}$ is a~scalar.


\begin{assumpV}\label{assump:EIVmisfit}
There exist 
$\phi:=\bar{\Sigma}_{\beps}-\bar{\bSigma}_{\btheta,\beps}^{\top}(\bar{\bSigma}_{\btheta}+\bbeta\bar{\bSigma}_{\btheta,\beps}^{\top})^{-1}(\bar{\bSigma}_{\btheta,\beps}+\bbeta\bar{\Sigma}_{\beps})\neq 0$ and $\upsilon:=\lim_{n\to\infty}n^{-1}\var\|\bY-\bX\bbeta\|_2^2>0$.
\end{assumpV}

Let us remark that $\bar{\bSigma}_{\btheta,\beps}=\zero$ for the uncorrelated error structure and, then, $\phi=\bar{\Sigma}_{\beps}$.

\subsection{SWIP}
Finally, the spectral weak invariance principle for the smallest eigenvalues is provided. Let us denote $\lambda_{i}:=\lambda_{min}(\bSigma^{-1}[\bX_{i},\bY_{i}]^{\top}[\bX_{i},\bY_{i}])$ for $2\leq i\leq n$, $\lambda_0:=\lambda_1:=0$ and $\widetilde{\lambda}_{i}:=\lambda_{min}(\bSigma^{-1}[\bX_{-i},\bY_{-i}]^{\top}[\bX_{-i},\bY_{-i}])$ for $0\leq i\leq n-2$, $\widetilde{\lambda}_n:=\widetilde{\lambda}_{n-1}:=0$. Note that $\lambda_n\equiv\widetilde{\lambda}_0$.

\begin{proposition}[SWIP]\label{prop:SWIP}
Let~\ref{eq:M} and~\ref{eq:H0} hold. Under Assumptions~\ref{assump:EIVdesign}, \ref{assump:EIVerrors}, and \ref{assump:EIVmisfit},
\begin{equation*}
\left\{\frac{1}{\sqrt{n}}\left(\lambda_{[nt]}-[nt]\sigma^2\right)\right\}_{t\in[0,1]}\xrightarrow[n\to\infty]{\dist[0,1]}\left\{\frac{\phi^2\upsilon}{1+\|\balpha\|_2^2}\mathcal{W}(t)\right\}_{t\in[0,1]}
\end{equation*}
and
\begin{equation*}
\left\{\frac{1}{\sqrt{n}}\left(\widetilde{\lambda}_{[n(1-t)]}-[n(1-t)]\sigma^2\right)\right\}_{t\in[0,1]}\xrightarrow[n\to\infty]{\dist[0,1]}\left\{\frac{\phi^2\upsilon}{1+\|\balpha\|_2^2}\widetilde{\mathcal{W}}(t)\right\}_{t\in[0,1]},
\end{equation*}
where $\{\mathcal{W}(t)\}_{t\in[0,1]}$ is a~standard Wiener process , $\widetilde{\mathcal{W}}(t)=\mathcal{W}(1)-\mathcal{W}(t)$, and $\balpha=(\bar{\bSigma}_{\btheta}+\bbeta\bar{\bSigma}_{\btheta,\beps}^{\top})^{-1}(\bar{\bSigma}_{\btheta,\beps}+\bbeta\bar{\Sigma}_{\beps})$.
\end{proposition}

\subsection{Extension to multivariate case}\label{subsec:extension}
Suppose that $\bbeta\in\mathbb{R}^{p\times q}$, $\bY\in\mathbb{R}^{n\times q}$, and $\beps\in\mathbb{R}^{n\times q}$, $q\geq 1$. Let the singular value decomposition (SVD) of the partial transformed data be
\[
[\bX_{[nt]},\bY_{[nt]}]\bSigma^{-1/2}=\bU(t)\bGamma(t)\bV^{\top}(t)=\sum_{i=1}^{p+q}\varsigma(t)^{(i)}\bu(t)^{(i)}\bv(t)^{(i)\top},
\]
where $\bu(t)^{(i)}$'s are the left-singular vectors, $\bv(t)^{(i)}$'s are the right-singular vectors, and $\varsigma(t)^{(i)}$'s are the singular values in the non-increasing order. One may replace $\lambda_{[nt]}$ by
\[
\Lambda_{[nt]}:=\sum_{j=1}^{q}\Big(\varsigma(t)^{(p+j)}\Big)^2
\]
in Proposition~\ref{prop:SWIP} (and analogously for~$\widetilde{\lambda}_{[n(1-t)]}$). Then, the SWIP can be derived again (see the proof of Proposition~\ref{prop:SWIP}), provided adequately extended assumptions on the errors $\{\beps_{n,1}\}_{n=1}^{\infty},\ldots,\{\beps_{n,q}\}_{n=1}^{\infty}$ instead of the original ones $\{\eps_n\}_{n=1}^{\infty}$. However, the consequent proofs would become more technical.

\section{Nuisance-parameter-free detection}\label{sec:test}
Consistent estimation of~$\bbeta$ can be performed via the generalized TLS approach~\citep{gallo,HV1989}. The optimizing problem
\begin{equation*}
[\bb,\hat{\btheta},\hat{\beps}]:=\mathop{\arg\min}_{\left[\btheta,\beps\right]\in\mathbb{R}^{n\times (p+1)},{\bbeta}\in\mathbb{R}^p}\left\|\left[\btheta,\beps\right]\bSigma^{-1/2}\right\|_F\quad\mbox{s.t.}\quad \bY-{\beps}=({\bX}-{\btheta}){\bbeta},
\end{equation*}
where $\|\cdot\|_F$ stands for the Frobenius matrix norm, has a~solution consisting of the estimator
\begin{equation}\label{eq:TLS}
\bb=(\bX^{\top}\bX-\lambda_{n}\bSigma_{\btheta})^{-1}(\bX^{\top}\bY-\lambda_{n}\bSigma_{\btheta,\beps})
\end{equation}
and the fitted errors $[\hat{\btheta},\hat{\beps}]$ such that
\begin{equation}\label{eq:Fnorm}
\big\|[\hat{\btheta},\hat{\beps}]\bSigma^{-1/2}\big\|_F^2=\lambda_n.
\end{equation}
We construct the changepoint test statistics based on property~\eqref{eq:Fnorm}.



\subsection{Changepoint test statistics}
Let us think of two TLS estimates of~$\bbeta$: The first one based on the first~$i$ data lines $[\bX_i,\bY_i]$ and the second one based on the first~$k$ data lines $[\bX_k,\bY_k]$ such that $1\leq i\leq k\leq n$. Under the null~\ref{eq:H0}, these two TLS estimates should be close to each other. On the other hand, under the alternative~\ref{eq:HA} such that $\tau\in\{i,\ldots,k\}$, they should be somehow different. A~similar conclusion can be made for the \emph{goodness-of-fit} statistics coming from~\eqref{eq:Fnorm}. It means that
\begin{equation*}
\lambda_i-\frac{i}{k}\lambda_k
\end{equation*}
should be reasonably small under the null~\ref{eq:H0}. Under the alternative~\ref{eq:HA} such that $\tau\in\{i,\ldots,k\}$, it should be relatively large. For the multivariate case described in previous Subsection~\ref{subsec:extension}, one has to replace~$\lambda_k$ by $\Lambda_k=\sum_{j=1}^{q}\big(\varsigma(k/n)^{(p+j)}\big)^2$.

We rely on \emph{self-normalized test statistics} introduced by~\cite{shao2010testing}, because the unknown quantity $\phi^2\upsilon/(1+\|\balpha\|_2^2)$ from Proposition~\ref{prop:SWIP} cancels out in the test statistics. Our \emph{supremum-type self-normalized test statistic} based on the \emph{goodness-of-fit} is defined as
\begin{equation}\label{eq:statisticS}
\mathscr{S}_n:=\max_{1\leq k< n}\frac{\big|\lambda_k-\frac{k}{n}\lambda_n\big|}{\max_{1\leq i< k}\big|\lambda_i-\frac{i}{k}\lambda_k\big|+\max_{k< i\leq n}\big|\widetilde{\lambda}_{i}-\frac{n-i}{n-k}\widetilde{\lambda}_{k}\big|}
\end{equation}
and the \emph{integral-type self-normalized test statistic} is defined as
\begin{equation}\label{eq:statisticT}
\mathscr{T}_n:=\sum_{k=1}^{n-1}\frac{\big(\lambda_k-\frac{k}{n}\lambda_n\big)^2}{\sum_{i=1}^{k-1}\big(\lambda_i-\frac{i}{k}\lambda_k\big)^2+\sum_{i=k+1}^{n}\big(\widetilde{\lambda}_{i}-\frac{n-i}{n-k}\widetilde{\lambda}_{k}\big)^2}.
\end{equation}

Let us note that evaluations of the above defined test statistics require just several singular value decompositions, which is reasonably \emph{quick}. Our new test statistics involve \emph{neither nuisance parameters nor tuning constants} and will work for non-stationary and weakly dependent data. On the top of that, \emph{no boundary issue} is present meaning that the tests can detect the change close to the beginning or to the end of the studied regime.

Under the null hypothesis and the technical assumptions from Subsection~\ref{subsec:assump}, the test statistics defined in~\eqref{eq:statisticS} and~\eqref{eq:statisticT} converge to \emph{non-degenerate limit distributions} (their quantiles can be found in Subsection~\ref{subsec:crit}).

\begin{thm}[Under the null]\label{thm:H0}
	Let~\ref{eq:M} and~\ref{eq:H0} hold. Under Assumptions~\ref{assump:EIVdesign}, \ref{assump:EIVerrors}, and \ref{assump:EIVmisfit},
	\begin{equation}\label{eq:limit_distS}
	\mathscr{S}_n\xrightarrow[n\to\infty]{\dist}\sup_{t\in[0,1]}\frac{\big|\mathcal{W}(t)-t\mathcal{W}(1)\big|}{\sup_{s\in[0,t]}\big|\mathcal{W}(s)-\frac{s}{t}\mathcal{W}(t)\big|+\sup_{s\in[t,1]}\big|\widetilde{\mathcal{W}}(s)-\frac{1-s}{1-t}\widetilde{\mathcal{W}}(t)\big|}
	\end{equation}
	and
	\begin{equation}\label{eq:limit_distT}
	\mathscr{T}_n\xrightarrow[n\to\infty]{\dist}\int_0^1\frac{\big\{\mathcal{W}(t)-t\mathcal{W}(1)\big\}^2}{\int_0^t\big\{\mathcal{W}(s)-\frac{s}{t}\mathcal{W}(t)\big\}^2\ud s+\int_t^1\big\{\widetilde{\mathcal{W}}(s)-\frac{1-s}{1-t}\widetilde{\mathcal{W}}(t)\big\}^2\ud s}\ud t,
	\end{equation}
	where $\{\mathcal{W}(t)\}_{t\in[0,1]}$ is a~standard Wiener process and $\widetilde{\mathcal{W}}(t)=\mathcal{W}(1)-\mathcal{W}(t)$.
\end{thm}

The null hypothesis is rejected at significance level $\alpha$ for large values of $\mathscr{S}_n$ and $\mathscr{T}_n$. The critical values can be obtained as the $(1-\alpha)$-quantiles of the asymptotic distributions from~\eqref{eq:limit_distS} and~\eqref{eq:limit_distT}. In order to describe limit behavior of the test statistics under the alternative, an~additional \emph{changepoint assumption} is required.

\begin{assumpC}\label{assump:EIVchange}
For some $\zeta\in(0,1)$, as $n\to\infty$,
\begin{equation}\label{eq:C}
\|\bdelta\|_2\to 0\quad\mbox{and}\quad (\eta\kappa-\bvarphi^{\top}\bvarphi)\sqrt{n}\to\infty,
\end{equation}
where $\kappa:=(\bar{\bSigma}_{\btheta,\beps}^{\top}+\bar{\Sigma}_{\beps}\bbeta^{\top})\bDelta_{\zeta}(\bar{\bSigma}_{\btheta,\beps}+\bbeta\bar{\Sigma}_{\beps})+(\bar{\bSigma}_{\btheta,\beps}^{\top}+\bar{\Sigma}_{\beps}(\bbeta+\bdelta)^{\top})\bDelta_{-\zeta}(\bar{\bSigma}_{\btheta,\beps}+(\bbeta+\bdelta)\bar{\Sigma}_{\beps})$, $\bvarphi:=(\bar{\bSigma}_{\btheta}+\bar{\Sigma}_{\btheta,\beps}\bbeta^{\top})\bDelta_{\zeta}(\bar{\bSigma}_{\btheta,\beps}+\bbeta\bar{\Sigma}_{\beps})+(\bar{\bSigma}_{\btheta}+\bar{\Sigma}_{\btheta,\beps}(\bbeta+\bdelta)^{\top})\bDelta_{-\zeta}(\bar{\bSigma}_{\btheta,\beps}+(\bbeta+\bdelta)\bar{\Sigma}_{\beps})$, and $\eta:=\lambda_{min}((\bar{\bSigma}_{\btheta}+\bar{\bSigma}_{\btheta,\beps}\bbeta^{\top})\bDelta(\bar{\bSigma}_{\btheta}+\bbeta\bar{\bSigma}_{\btheta,\beps}^{\top})+\sigma^2\bI_p)-\sigma^2$.
\end{assumpC}

This assumption may be considered as a~changepoint \emph{detectability requirement} for local alternatives, because it manages the relationship between the size of the change, the location of the change, and the noisiness of the data in order to be able to detect the changepoint. In case of uncorrelated error structure, the previous formulae become simpler due to $\bar{\bSigma}_{\btheta,\beps}=\zero$. Assumption~\ref{assump:EIVchange} is automatically fulfilled, for instance, for an~arbitrary $\delta\to 0$ and the one-dimensional equidistant design points $Z_i$'s on $(0,1)$ with homoscedastic error structure, because then $\eta\kappa-\bvarphi^{\top}\bvarphi=\beta^2\{\zeta^3+(1-\zeta)^3\}\{1-\zeta^3-(1-\zeta)^3\}/9+O(\delta)$ as $\delta\to 0$. Furthermore, let us remark that $\bvartheta:=\bar{\bSigma}_{\btheta}+\bbeta\bar{\bSigma}_{\btheta,\beps}^{\top}$ has full rank under Assumption~\ref{assump:EIVmisfit}. 


Now, the tests based on $\mathscr{S}_n$ and $\mathscr{T}_n$ are shown to be \emph{consistent}, as the test statistics converge to infinity under some local alternatives, provided that the size of the change does not convergence to zero too fast, cf.~Assumption~\ref{assump:EIVchange} where $\kappa$ and $\bvarphi$ depend on~$\bdelta$.

\begin{thm}[Under local alternatives]\label{thm:H1}
Let~\ref{eq:M} and~\ref{eq:HA} hold such that $\tau=[n\zeta]$ for some $\zeta\in(0,1)$. Under Assumptions~\ref{assump:EIVchange}, \ref{assump:EIVdesign}, \ref{assump:EIVerrors}, and \ref{assump:EIVmisfit},
\begin{equation}\label{eq:alt}
\mathscr{S}_n\xrightarrow[n\to\infty]{\prob}\infty\xleftarrow[n\to\infty]{\prob}\mathscr{T}_n.
\end{equation}
\end{thm}

Assumption~\ref{assump:EIVchange} can be sharpened as remarked below with the corresponding proof in the Appendix~\ref{sec:proofs}.

\begin{remark}\label{rmk:sharp}
The second part of relation~\eqref{eq:C} can be replaced by
\begin{equation}\label{eq:Csharp}
\{\kappa+\eta-\sqrt{(\kappa+2\sigma^2+\eta)^2-4(\kappa+\sigma^2-\bvarphi^{\top}(\bvartheta^{\top}\bDelta\bvartheta+\sigma^2\bI_p)^{-1}\bvarphi)(\sigma^2+\eta)}\}\sqrt{n}\to\infty
\end{equation}
and the assertion of Theorem~\ref{thm:H1} still holds.
\end{remark}

Basically, Theorem~\ref{thm:H1} discloses that in presence of the structural change in linear relations, the test statistics \emph{explode above all bounds}. Hence, the asymptotic distributions from Theorem~\ref{thm:H0} can be used to construct the tests. Although, explicit forms of those distributions stated in~\eqref{eq:limit_distS} and~\eqref{eq:limit_distT} are unknown.

\subsection{Asymptotic critical values}\label{subsec:crit}
The critical values may be determined by simulations from the limit distributions $\mathscr{S}_n$ and $\mathscr{T}_n$ from Theorem~\ref{thm:H0}. Theorem~\ref{thm:H1} ensures that we reject the null hypothesis for large values of the test statistics. We have simulated the asymptotic distributions~\eqref{eq:limit_distS} and~\eqref{eq:limit_distT} by \emph{discretizing} the standard Wiener process and using the relationship of a~random walk to the standard Wiener process. We considered $1000$ as the number of discretization points within $[0,1]$ interval and the number of simulation runs equals to $100000$. In Table~\ref{tab:crit_val}, we present several critical values for the test statistics~$\mathscr{S}_n$ and~$\mathscr{T}_n$.

\begin{table}
	\caption{\label{tab:crit_val}Simulated asymptotic critical values for~$\mathscr{S}_n$ and~$\mathscr{T}_n$}
	\centering
		\begin{tabular}{rrrrrr}
			\toprule
			$100(1-\alpha)\%$ & $90\%$ & $95\%$ & $97.5\%$ & $99\%$ & $99.5\%$ \\
			\midrule
			$\mathscr{S}$-based & $1.209008$ & $1.393566$ & $1.571462$ & $1.782524$ & $1.966223$\\
			$\mathscr{T}$-based & $5.700222$ & $7.165705$ & $8.807070$ & $10.597625$ & $11.755233$\\
			\bottomrule
		\end{tabular}
\end{table}

\subsection{Changepoint estimator}
If a~change is \emph{detected}, it is of interest to estimate the changepoint. It is sensible to use
\begin{equation*}
\hat{\tau}_n:=\mathop{\operatorname{argmax}}_{1\leq k \leq n-1}\frac{\big|\lambda_k-\frac{k}{n}\lambda_n\big|+\big|\widetilde{\lambda}_k-\frac{n-k}{n}\widetilde{\lambda}_0\big|}{\max_{1\leq i< k}\big|\lambda_i-\frac{i}{k}\lambda_k\big|+\max_{k< i\leq n}\big|\widetilde{\lambda}_{i}-\frac{n-i}{n-k}\widetilde{\lambda}_{k}\big|}
\end{equation*}
as a~\emph{changepoint estimator}. Our next theorem shows that under the alternative, the changepoint $\tau$ is consistently estimated by the estimator $\hat{\tau}_n$.

\begin{corollary}[Consistency]\label{cor:est}
Let the assumptions of Theorem~\ref{thm:H1} hold. If
\begin{align}
&\forall t\in(\zeta,1):\,\{\eta(t)\kappa(t)-\bvarphi(t)^{\top}\bvarphi(t)\}\sqrt{n}\xrightarrow{n\to\infty}\infty;\label{eq:Cpt}\\
&\forall t\in(0,\zeta):\,\{\tilde{\eta}(t)\tilde{\kappa}(t)-\tilde{\bvarphi}(t)^{\top}\tilde{\bvarphi}(t)\}\sqrt{n}\xrightarrow{n\to\infty}\infty,\label{eq:Cpt2}
\end{align}
where $\kappa(t):=(\bar{\bSigma}_{\btheta,\beps}^{\top}+\bar{\Sigma}_{\beps}\bbeta^{\top})\bDelta_{\zeta}(\bar{\bSigma}_{\btheta,\beps}+\bbeta\bar{\Sigma}_{\beps})+(\bar{\bSigma}_{\btheta,\beps}^{\top}+\bar{\Sigma}_{\beps}(\bbeta+\bdelta)^{\top})(\bDelta_{t}-\bDelta_{\zeta})(\bar{\bSigma}_{\btheta,\beps}+(\bbeta+\bdelta)\bar{\Sigma}_{\beps})$, $\bvarphi(t):=(\bar{\bSigma}_{\btheta}+\bar{\Sigma}_{\btheta,\beps}\bbeta^{\top})\bDelta_{\zeta}(\bar{\bSigma}_{\btheta,\beps}+\bbeta\bar{\Sigma}_{\beps})+(\bar{\bSigma}_{\btheta}+\bar{\Sigma}_{\btheta,\beps}(\bbeta+\bdelta)^{\top})(\bDelta_{t}-\bDelta_{\zeta})(\bar{\bSigma}_{\btheta,\beps}+(\bbeta+\bdelta)\bar{\Sigma}_{\beps})$, $\tilde{\kappa}(t):=(\bar{\bSigma}_{\btheta,\beps}^{\top}+\bar{\Sigma}_{\beps}\bbeta^{\top})\bDelta_{-\zeta}(\bar{\bSigma}_{\btheta,\beps}+\bbeta\bar{\Sigma}_{\beps})+(\bar{\bSigma}_{\btheta,\beps}^{\top}+\bar{\Sigma}_{\beps}(\bbeta+\bdelta)^{\top})(\bDelta_{-t}-\bDelta_{-\zeta})(\bar{\bSigma}_{\btheta,\beps}+(\bbeta+\bdelta)\bar{\Sigma}_{\beps})$, $\tilde{\bvarphi}(t):=(\bar{\bSigma}_{\btheta}+\bar{\Sigma}_{\btheta,\beps}\bbeta^{\top})\bDelta_{-\zeta}(\bar{\bSigma}_{\btheta,\beps}+\bbeta\bar{\Sigma}_{\beps})+(\bar{\bSigma}_{\btheta}+\bar{\Sigma}_{\btheta,\beps}(\bbeta+\bdelta)^{\top})(\bDelta_{-t}-\bDelta_{-\zeta})(\bar{\bSigma}_{\btheta,\beps}+(\bbeta+\bdelta)\bar{\Sigma}_{\beps})$, $\eta(t):=\lambda_{min}((\bar{\bSigma}_{\btheta}+\bar{\bSigma}_{\btheta,\beps}\bbeta^{\top})\bDelta_t(\bar{\bSigma}_{\btheta}+\bbeta\bar{\bSigma}_{\btheta,\beps}^{\top})+t\sigma^2\bI_p)-t\sigma^2$, and $\tilde{\eta}(t):=\lambda_{min}((\bar{\bSigma}_{\btheta}+\bar{\bSigma}_{\btheta,\beps}\bbeta^{\top})\bDelta_{-t}(\bar{\bSigma}_{\btheta}+\bbeta\bar{\bSigma}_{\btheta,\beps}^{\top})+(1-t)\sigma^2\bI_p)-(1-t)\sigma^2$, then
\[
\frac{\hat{\tau}_n}{n}\xrightarrow[n\to\infty]{\prob}\zeta.
\]
\end{corollary}


Conditions~\eqref{eq:Cpt} and~\eqref{eq:Cpt2} serve as a~uniform intermediary between the size of the change, the location of the change, the sample size, and the heteroscedasticity of the disturbances for assuring changepoint estimator's consistency. These assumptions are again automatically fulfilled for the case discussed below Assumption~\ref{assump:EIVchange}.

In order to estimate more than one changepoint, it is possible to use an~arbitrary `divide-and-estimate' \emph{multiple changepoints} method relying on our changepoint estimator, for instance, wild binary segmentation by~\cite{F2014}.

\section{Simulation study}\label{sec:simulations}
We are interested in the performance of the tests based on the self-normalized test statistics $\mathscr{S}_n$ and $\mathscr{T}_n$ that are completely nuisance-parameter-free. We focused on the comparison of the \emph{accuracy of critical values} obtained by the simulation from the limit distributions.

In Figures~\ref{fig:H0}--\ref{fig:H12D}, one may see \emph{size-power plots} considering the test statistics $\mathscr{S}_n$ and $\mathscr{T}_n$ under the null hypothesis and under the alternative. Figures~\ref{fig:H0} and~\ref{fig:H1} correspond to one input covariate (i.e., $p=1$) with choices of $\beta=1$ and $Z_{i,1}=100i/(n+1)$. A~case with two error-prone regressors (i.e., $p=2$) is illustrated in Figures~\ref{fig:H02D} and~\ref{fig:H12D} for choices of $\bbeta=[1,1]^{\top}$ and $Z_{i,\bullet}=100\times[i/(n+1),(i/(n+1))^{3/2}]$. Next, $n\in\{200,1000\}$ and $\tau\in\{n/4,n/2\}$. The size of change is $\delta\in\{0.1,0.5\}$ for $p=1$ and $\bdelta\in\{[0.1,0.1]^{\top},[0.5,0.5]^{\top}\}$ for $p=2$. Especially smaller values of the break should represent the situations under the local alternatives. In Figures~\ref{fig:H0} and~\ref{fig:H02D}, the empirical rejection frequency under the null hypothesis (actual $\alpha$-errors) is plotted against the theoretical size (theoretical $\alpha$-errors with $\alpha\in\{1\%,5\%,10\%\}$), illustrating the size of the tests. The ideal situation under the null hypothesis is depicted by the straight diagonal dotted line. The empirical rejection frequencies ($1-$errors of the second type) under the alternative (with different changepoints and values of the change) are shown in Figures~\ref{fig:H1} and~\ref{fig:H12D}, illustrating the power of the tests. Under the alternative, the desired situation would be a~steep function with values close to~1. For more details on the size-power plots we may refer, e.g., to~\cite{kir2006}. The standard deviation of the random disturbances was set to $\sigma\in\{0.5,1.0\}$ and the random error terms $\{\Theta_{n,1}\}_{n=1}^{\infty},\ldots,\{\Theta_{n,p}\}_{n=1}^{\infty}$, and $\{\eps_n\}_{n=1}^{\infty}$ were independently simulated as three time series:
\begin{itemize}
	\setlength\itemsep{-0.1cm}
	\item \textsf{IID} \ldots~independent and identically distributed random variables;
	\item \textsf{AR(1)} \ldots~autoregressive (AR) process of order one having a~coefficient of autoregression equal $0.5$;
	\item \textsf{ARCH(1)} \ldots~autoregressive conditional heteroscedasticity (ARCH) process with the second coefficient equal $0.5$.
\end{itemize}

The standard normal distribution and the Student $t$-distribution with 3~degrees of freedom are used for generating the innovations of the models' errors. All of the time series are standardized such that they have variance equal~$\sigma^2$. Let us remark that the setup of Student $t_3$-distribution does not satisfy Assumption~\ref{assump:EIVerrors}. However, it can be considered as a~misspecified model and one would like to inspect performance of our procedures on such a~model that violates our assumptions. In the simulations of the rejection rates, we used $10000$ repetitions.

\begin{figure}[!ht]
	\centering
	\includegraphics[width=0.9\textwidth]{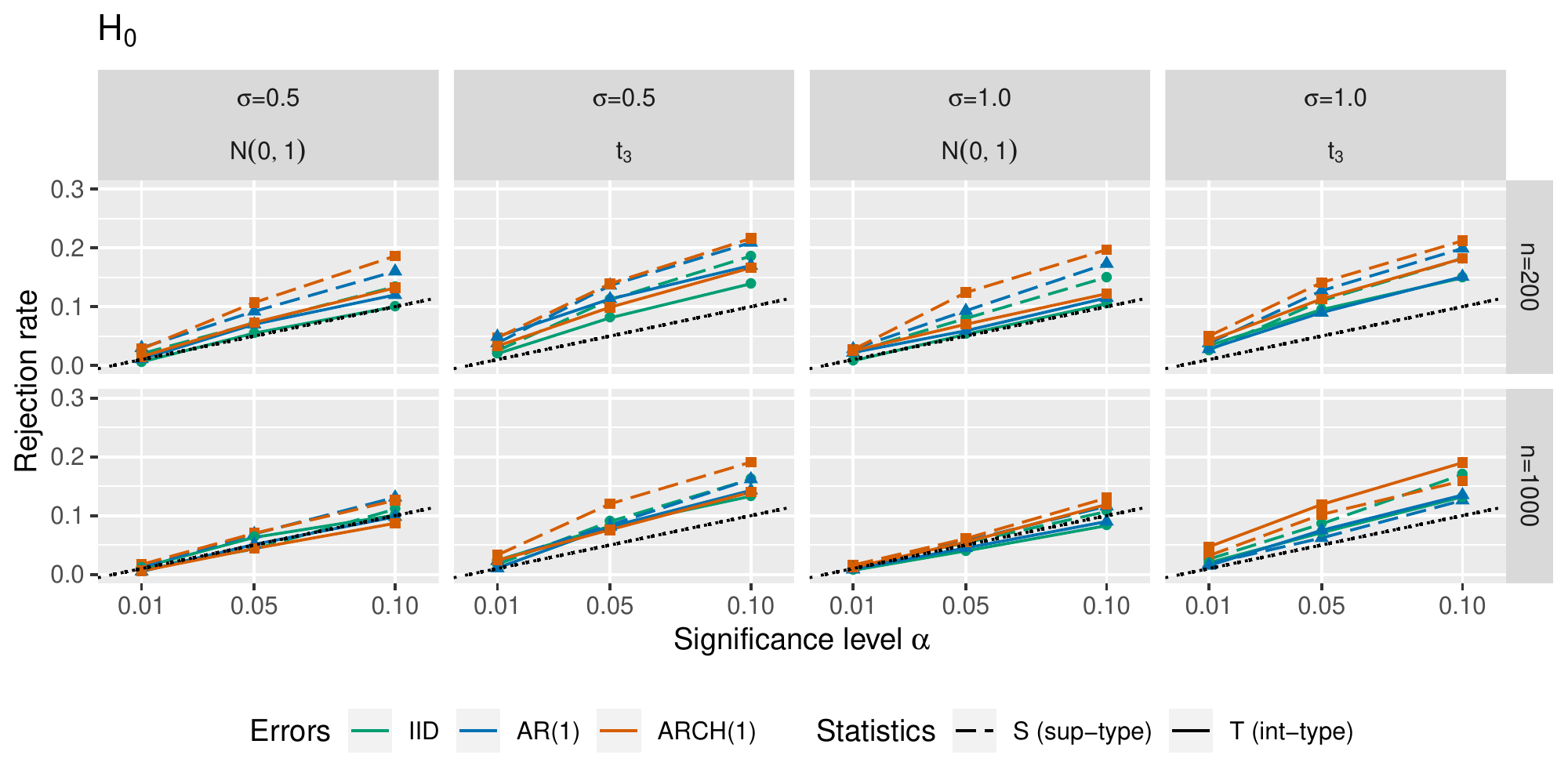}
	\caption{\label{fig:H0}Size-power plots for $\mathscr{S}_n$ and $\mathscr{T}_n$ under~\ref{eq:H0} ($p=1$)}
\end{figure}

\begin{figure}[!ht]
	\centering
		\includegraphics[width=0.9\textwidth]{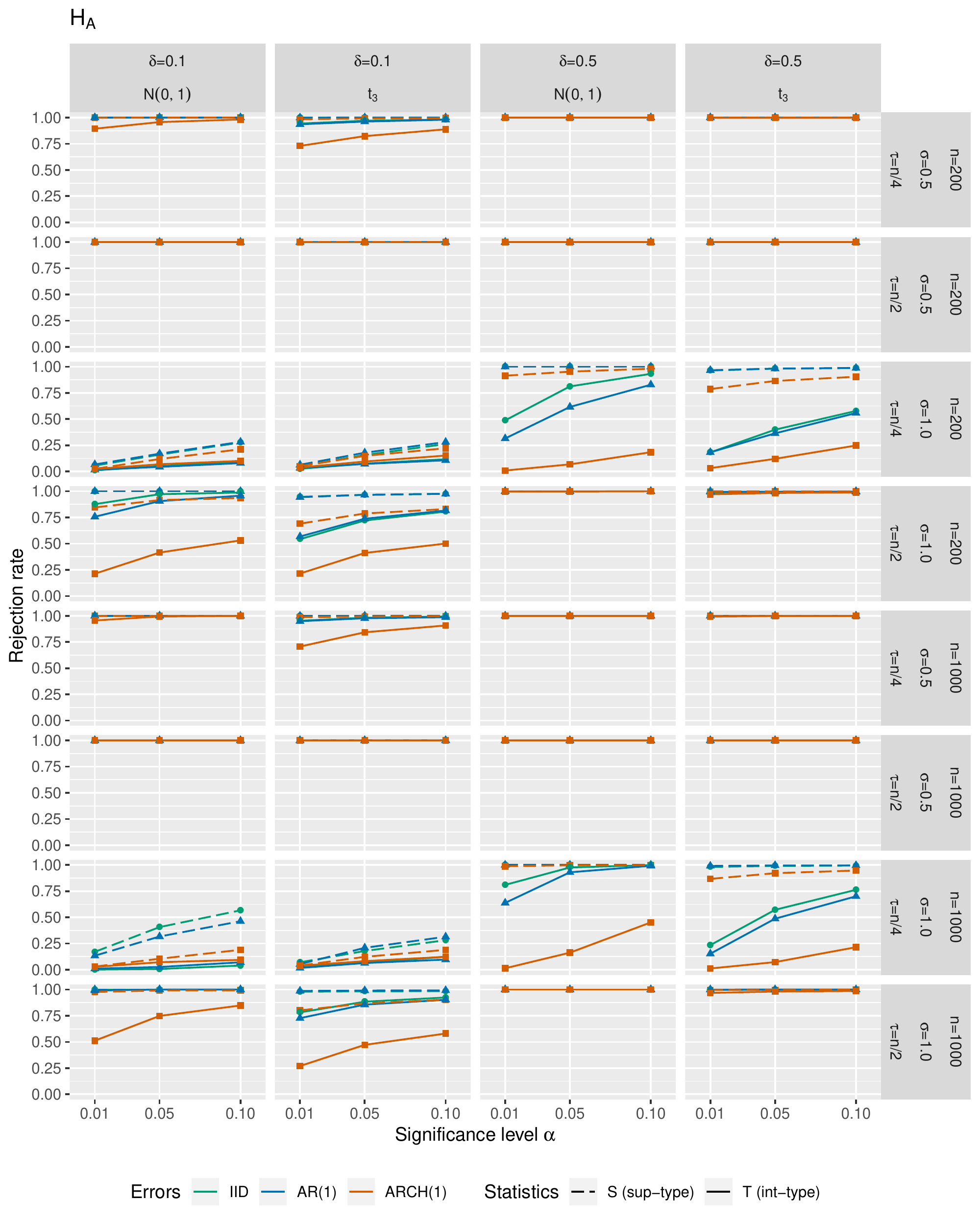}
		\caption{\label{fig:H1}Size-power plots for $\mathscr{S}_n$ and $\mathscr{T}_n$ under~\ref{eq:HA} ($p=1$)}
\end{figure}

\begin{figure}[!ht]
	\centering
	\includegraphics[width=0.9\textwidth]{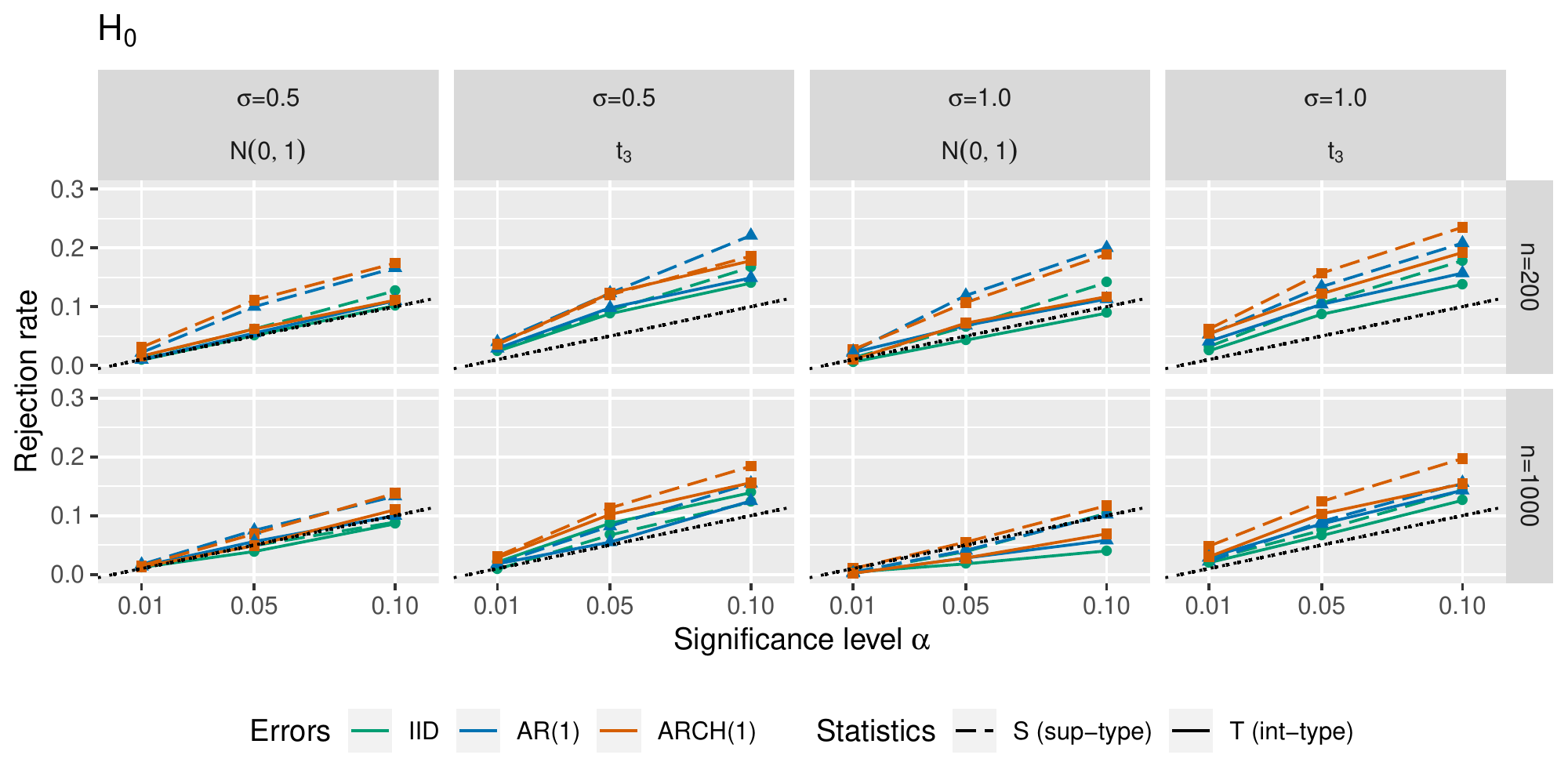}
	\caption{\label{fig:H02D}Size-power plots for $\mathscr{S}_n$ and $\mathscr{T}_n$ under~\ref{eq:H0} ($p=2$)}
\end{figure}

\begin{figure}[!ht]
	\centering
	\includegraphics[width=0.9\textwidth]{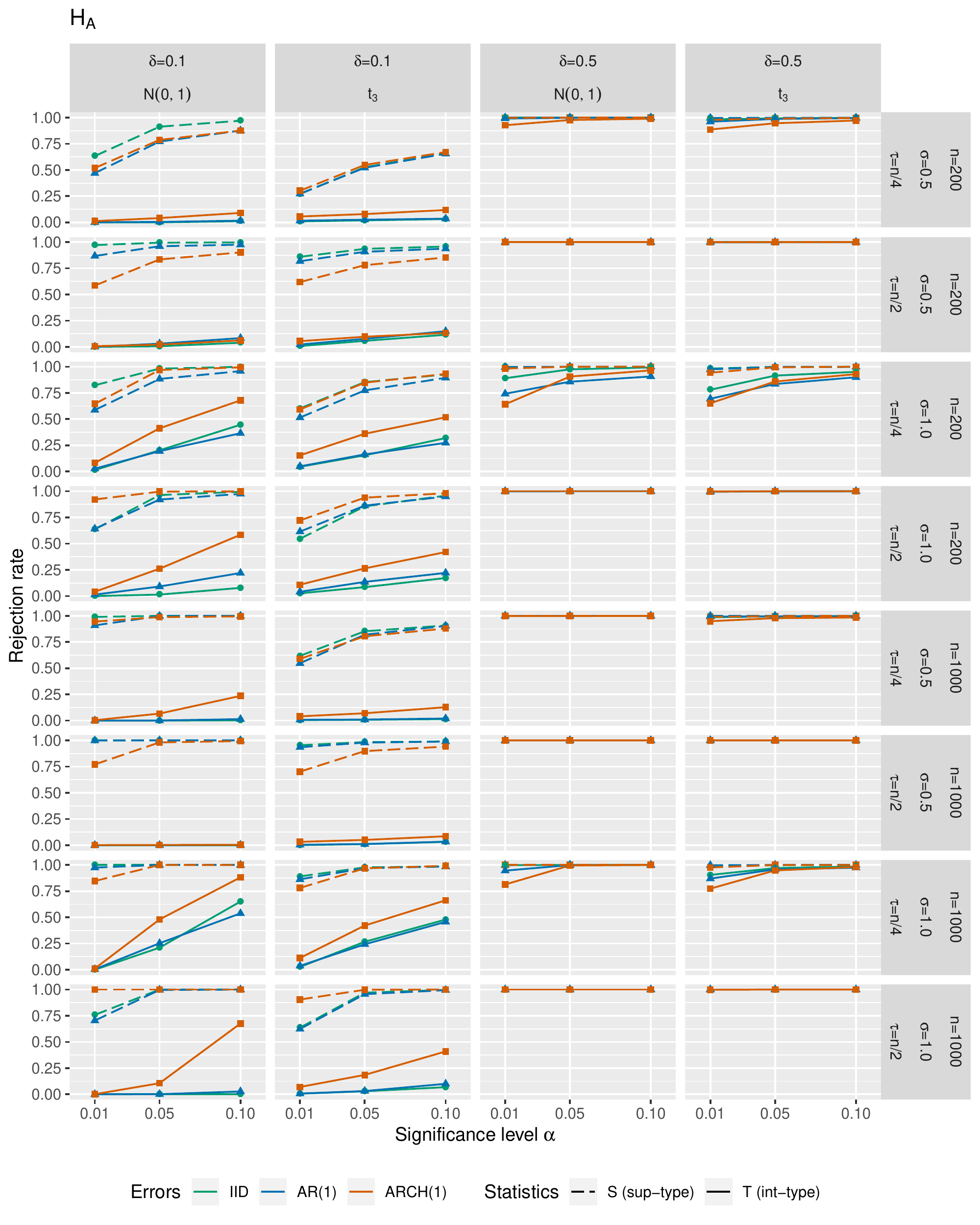}
	\caption{\label{fig:H12D}Size-power plots for $\mathscr{S}_n$ and $\mathscr{T}_n$ under~\ref{eq:HA} ($p=2$)}
\end{figure}

In all of the subfigures of Figures~\ref{fig:H0} and~\ref{fig:H02D} depicting a~situation under the null hypothesis, we may see that comparing the accuracy of $\alpha$-levels (sizes) for different self-normalized test statistics, the integral-type ($\mathscr{T}$-based) method seems to keep the theoretical significance level more firmly than the supre\-mum-type ($\mathscr{S}$-based) method. Comparing the case of $\mathsf{N}(0,1)$ innovations with the case of~$t_3$ innovations, the rejection rates under the null tend to be slightly higher for the~$t_3$-distribution. In spite of the fact that the $t_3$-distributed errors violate Assumption~\ref{assump:EIVerrors}, the performance of our tests is still surprisingly satisfactory in such case. As expected, the accuracy of the critical values tends to be better for larger~$n$. The more complicated dependence structure of errors is assumed, the worse performance of the tests is obtained. Furthermore, the less volatile errors are set, the better tests' sizes are attained.

The $\mathscr{T}$-method performs better under the null. However under the alternative, the $\mathscr{S}$-method has a~tendency to have slightly higher power than the $\mathscr{T}$-method (see Figures~\ref{fig:H1} and~\ref{fig:H12D}). We may also conclude that under~\ref{eq:HA} with less volatile errors, the power of the test increases. The power decreases when the changepoint is closer to the beginning or the end of the input-output data. The heavier tails ($t_3$ against $\mathsf{N}(0,1)$) give worse results in general for both test statistics. Moreover, `more dependent' scenarios reveal worsening of the test statistics' performance. Furthermore, the smaller size of the change is considered, the lower power of the test is achieved. And again, the power gets higher for larger~$n$.




Afterwards, a~simulation experiment is performed to study the \emph{finite sample} properties of the changepoint estimator for a~change in the linear relations' parameter. We numerically present only the case of $p=1$. In particular, the interest lies in the \emph{empirical distributions} of the proposed estimator visualized via boxplots, see Figure~\ref{fig:Estimator}. The simulation setup is kept the same as described above.

\begin{figure}[!ht]
	\centering
		\includegraphics[width=0.9\textwidth]{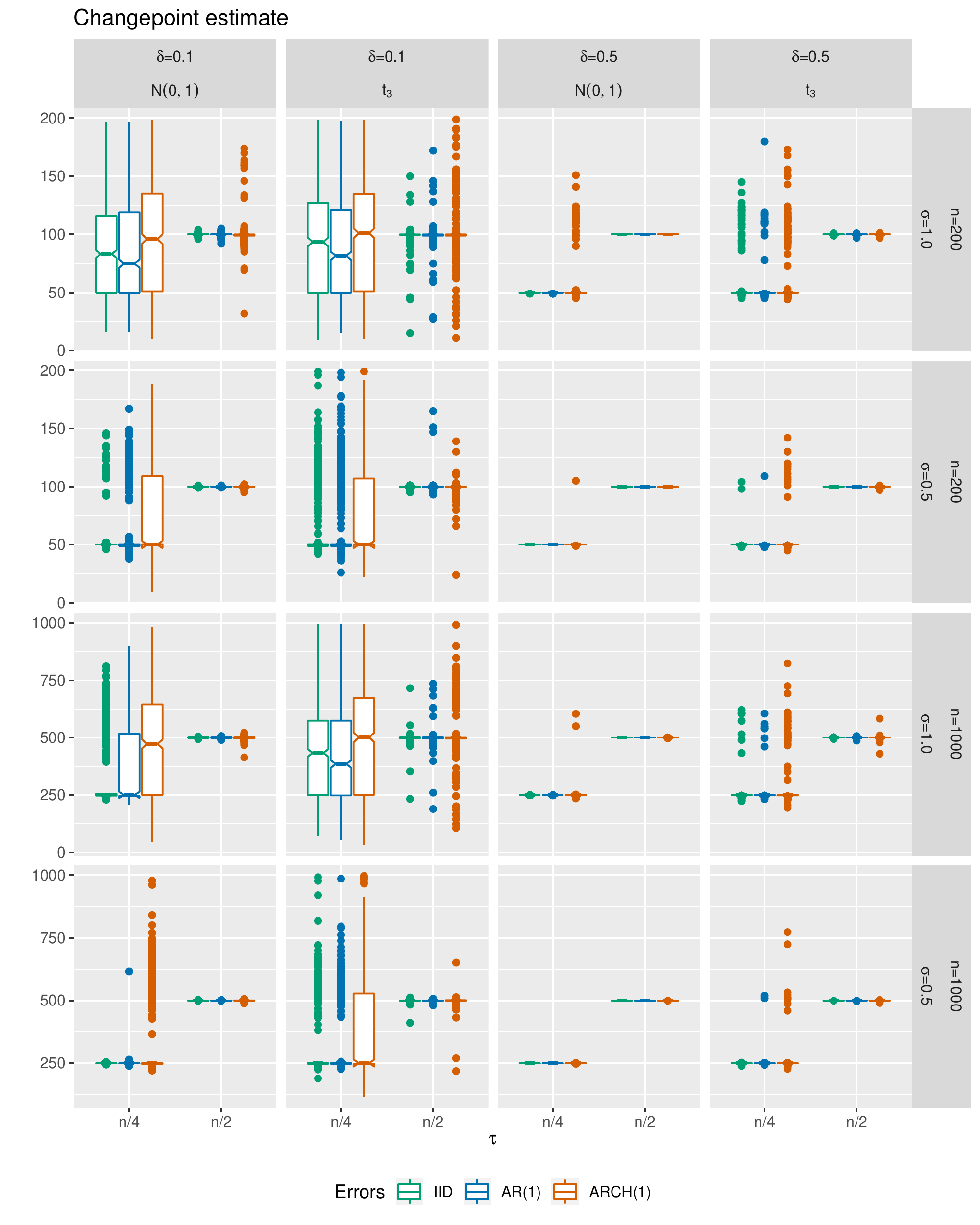}
		\caption{\label{fig:Estimator}Boxplots of the estimated changepoint $\hat{\tau}_n$ ($p=1$)}
\end{figure}

It can be concluded that the precision of our changepoint estimate is satisfactory even for relatively small sample sizes regardless of the errors' structure. Less volatile model errors provide more precise changepoint estimate. The less complicated dependence structure is assumed, the higher accuracy of the estimator is obtained. Furthermore, the disturbances with heavier tails yield less precise estimates than innovations with light tails. One may notice that higher precision is obtained when the changepoint is closer to the middle of the data. It is also clear that the precision of $\hat{\tau}_n$ improves markedly as the size of change increases.

\section{Applications}\label{sec:applications}

\subsection{Practical application: Calibration}\label{subsec:Papplication}
A~company has two industrial devices, where the first one is calibrated according to some institute of standards and the second one is just a~casual device. We want to test whether the second device is calibrated according to the first one. In this \emph{calibration problem}, it means to know whether the second device has approximately the \emph{same performance up to some unknown multiplication constant} as the first one. Consequently, other devices of the same type are needed to be calibrated as well. For some reasons, e.g., economic or logistic, it is only possible to calibrate one device by the official authorities.

Our data set, provided by a~Czech steelmaker, contains 100 couples of speed values of two hammer rams (see Figure~\ref{fig:Presser}), where the first forging hammer is calibrated. We set the same power level on both hammers and measure the speed of each hammer ram repeatedly changing only the power level. Our measurements of the speed are encumbered with errors of the same variability in both cases, because we use the same device for measuring the speed and both forging hammers are of the same type. Since the power set for the forging hammer is directly proportional to the speed of the hammer ram, our goal is to test whether the \emph{ratio of two hammer rams' speeds is kept constant} over changing the power level or not. Therefore, our changepoint in the EIV model is very suitable for this setup---a~linear dependence and errors in both measured speeds (with the same variance).

Both our changepoint tests---$\mathscr{S}_n=83.2$ and $\mathscr{T}_n=861.4$---reject the null hypothesis of a~constant linear coefficient between two hammer rams' speed values at the significance level of $\alpha=0.5\%$ (cf.~Table~\ref{tab:crit_val}; the significance level for technical fields is usually smaller than the standard $5\%$), indicating a~changed performance of the second non-calibrated hammer ram.

\begin{figure}[!ht]
	\begin{center}
		\includegraphics[width=0.8\textwidth]{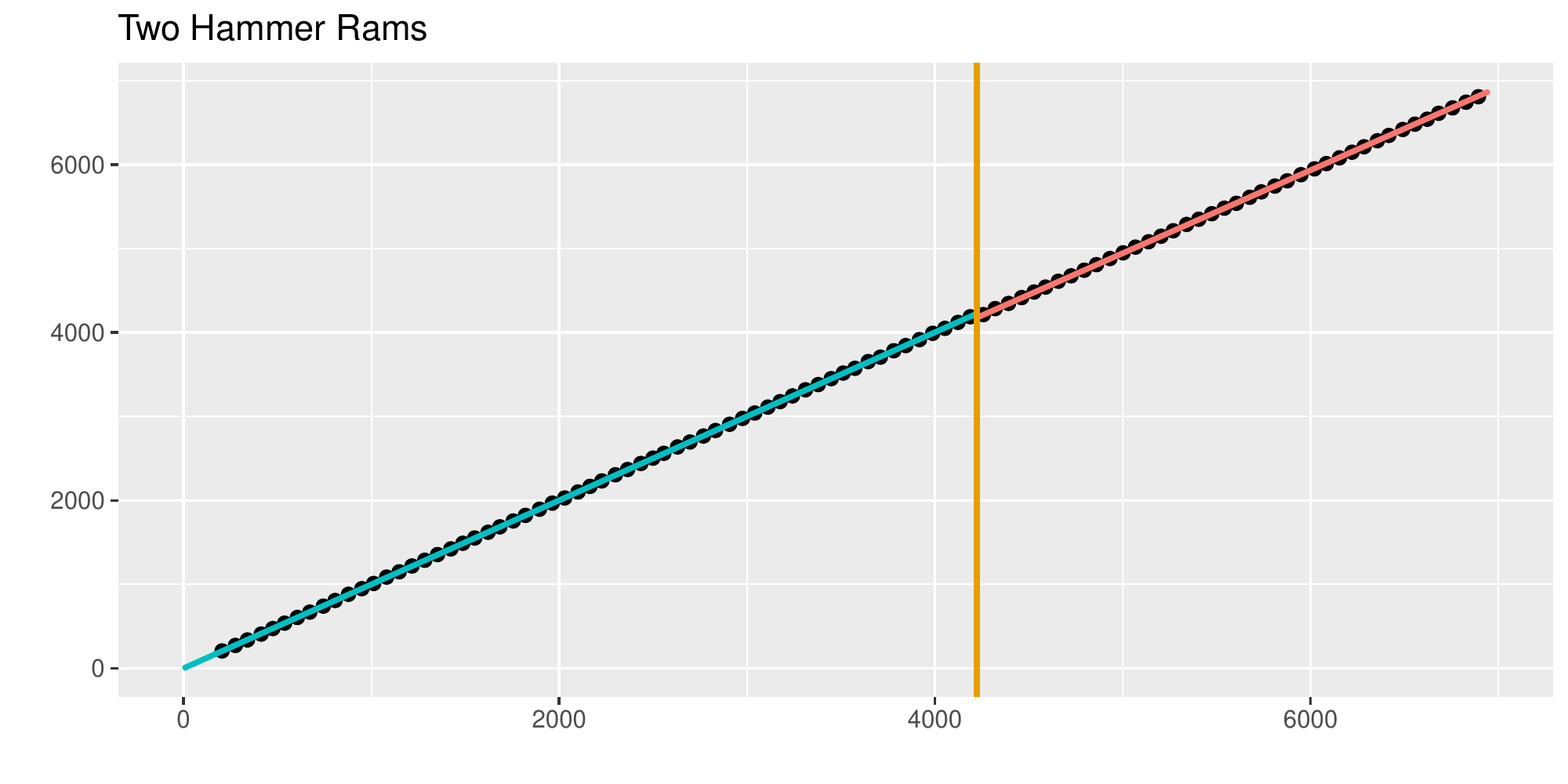}
		\caption{Speeds of two hammer rams, where the first one displayed on the x-axis is calibrated. The changepoint estimate corresponding to the technical issues of the second hammer ram after the 60th measurement is depicted by the vertical line}
		\label{fig:Presser}
	\end{center}
\end{figure}

As an~estimate for our change, we obtain $\hat{\tau}_n=60$ (depicted by a~vertical line in Figure~\ref{fig:Presser}), which corresponds to the 60th measurement of pair of speeds. After this particular measurement, we have background information that a~technical issue appeared to the second hammer ram---one of its oil tubes started to leak. Our procedure is indeed capable to detect and, consequently, to estimate the changepoint in the ratio of the hammer rams' speeds. And this is done fully automatically without expert knowledge about the oil tube issue and also without setting tuning parameters. Moreover, the estimated ratio via the TLS approach before the change is $1.000891$ (the slope of the green line in Figure~\ref{fig:Presser}), which basically says that the hammer rams work approximately in the same way. However, the estimated ratio via the TLS approach after the change is $0.9892154$ (the slope of the red line in Figure~\ref{fig:Presser}), which is significantly different from constant~1 (see a~formal statistical test by~\cite{P2013}).

Other calibration examples, where our methodology is applicable, can be found in, e.g., \cite{ChR2006} or~\cite{GL2011}.

\subsection{Theoretical application: Randomly spaced time series}\label{subsec:Tapplication}
A~motivation for the changepoint problem in randomly spaced time series comes from the changepoint in the polynomial \emph{trending regression}~\citep{AHH2009}. Let us think of a~single regressor measured precisely such that $X_{i,1}\equiv Z_{i,1}=i/(n+1)$. This indeed corresponds to a~situation of a~one-dimensional equally (regularly) spaced time series, where the original time points $\{i\}_{i=1}^n$ are `squeezed' into the interval~$[0,1]$ by dividing of~$n+1$.

Now, let us assume that our outcome observations $Y_i$'s are supposed to be measured at some unknown time points~$Z_{i,1}$'s. However, due to some measurement imprecision or some outer random influence, the actual observation~$Y_i$, which should correspond to~$Z_{i,1}$, is not recorded at time point~$Z_{i,1}$, but at time point~$X_{i,1}$. One can imagine a~long-distance time trial against the chronometer (e.g., an~individual competition in cross-country skiing). There are $n$ intermediate spots on the track, where the athlete's time is recorded. If we think of one particular athlete, we measure at the intermediate spot~$i$ her/his error-prone competition time $X_{i,1}$, which was encumbered by some randomness, instead of the true unobservable time is $Z_{i,1}$. This is because each race is specific and every athlete has a~unique performance during that particular race. We also observe a~time lag~$Y_{i}$ between her/him and the current leader at that spot. Now, one is interested whether there is a~change in linear trend. This would help to analyze whether the particular athlete tried to improve or not during the time trial. One can argue that a~distance of the intermediate spot should be taken into account instead of the athlete's intermediate time. However, the intermediate distance is also measured with error, for instance, a~rounded value of the true unobserved distance is provided. Another example of randomly spaced time series is a~case when the observation times are driven by the series itself. For instance, cumulative counts of occurrences of a~disease in a~given area~\citep{W1986}.

The unobservable sequence $\{Z_{i,1}\}_{i=1}^n$ can be regularly or irregularly spaced. The key issue is to have satisfiable Assumption~\ref{assump:EIVdesign}. Since the developed detection procedures rely on the orthogonal regression, it is sufficient to transform the original randomly spaced time series~$\{X_{i,1},Y_i\}_{i=1}^n$ into, e.g., $\{X_{i,1}/(\max_i\{|X_{i,1}|\}+\epsilon),Y_i/(\max_i\{|X_{i,1}|\}+\epsilon)\}_{i=1}^n$, where a~constant~$\epsilon$ is reasonably large. Afterwards, the proposed tests remain valid when applied to the transformed randomly spaced time series, because~$\beta$ stays unchanged after such a~transformation. Hence, one can test whether the \emph{linear trend} has or has not changed over time.


\section{Conclusions}
Our changepoint problem in linear relations is linearly defined, but comes with a~highly non-linear solution and inference. We have proposed two tests for changepoints with desirable theoretical properties: The asymptotic size of the tests is guaranteed by a~limit theorem even under non-stationarity and weak dependency, the tests and the related changepoint estimator are consistent. We are not aware of any similar results even for independent and identically distributed errors. By combining self-normalization and the proposed spectral weak invariance principle, there are neither tuning constants nor nuisance parameters involved in the whole testing procedure. Therefore, the detection methods are completely data-driven, which makes this framework effortlessly applicable as demonstrated. In our simulations, the tests show reliable performance.


\appendix

\section{Proofs}\label{sec:proofs}

\begin{proof}[Proof of Proposition~\ref{prop:SWIP}]
Let the singular value decomposition of the transformed `partial' data matrix be
\[
[\bX_{[nt]},\bY_{[nt]}]\bSigma^{-1/2}=\bU(t)\bGamma(t)\bV(t)^{\top}=\sum_{i=1}^{p+1}\varsigma(t)^{(i)}\bu(t)^{(i)}\bv(t)^{(i)\top}
\]
for some $t\in(0,1]$. Note that we are in a~situation of no change in the parameter~$\bbeta$. Bearing in mind Assumptions~\ref{assump:EIVdesign} and~\ref{assump:EIVerrors}, \citet[Lemma~2.1]{gleser} and~\citet[Theorem~3.1]{P2011} provide that $0\neq v_{p+1}(t)^{(p+1)}$ (i.e., the last element of the last right-singular vector $\bv(t)^{(p+1)}$ corresponding to the smallest singular value) with probability tending to one as~$n$ increases. 
According to~\citet[proof of Lemma~4.2]{gleser}, one gets
\begin{align}
&\frac{1}{\sqrt{n}}\left(\lambda_{[nt]}-[nt]\sigma^2\right)=\Big(v_{p+1}(t)^{(p+1)}\Big)^2[\ba_{t}^{\top},-1]\left\{\frac{1}{\sqrt{n}}\left(\bD_{t}-\E\bD_{t}\right)\right\}\begin{bmatrix}
\ba_{t}\\
-1
\end{bmatrix}\label{eq:swip1}\\
&\quad+\Big(v_{p+1}(t)^{(p+1)}\Big)^2\sqrt{n}
[\ba_{t}^{\top},-1]\bSigma^{-1/2}\begin{bmatrix}
\bI_{p}\\
\bbeta^{\top}
\end{bmatrix}\frac{1}{n}\bZ_{[nt]}^{\top}\bZ_{[nt]}[\bI_{p},\bbeta]\bSigma^{-1/2}\begin{bmatrix}
\ba_{t}\\
-1
\end{bmatrix},\label{eq:swip2}
\end{align}
where $\ba_{t}:=(\tilde{\bX}_{[nt]}^{\top}\tilde{\bX}_{[nt]}-\lambda_{[nt]}\bI_{p})^{-1}\tilde{\bX}_{[nt]}^{\top}\tilde{\bY}_{[nt]}$ is the TLS estimator for the transformed data $[\tilde{\bX}_{[nt]},\tilde{\bY}_{[nt]}]:=[\bX_{[nt]},\bY_{[nt]}]\bSigma^{-1/2}$ and $\bD_{t}:=\bSigma^{-1/2}[\bX_{[nt]},\bY_{[nt]}]^{\top}[\bX_{[nt]},\bY_{[nt]}]\bSigma^{-1/2}$. With respect to~\cite{P2011}, we have
\begin{equation*}
\Big(v_{p+1}(t)^{(p+1)}\Big)^2=1-\left\|[v_{1}(t)^{(p+1)},\ldots,v_{p}(t)^{(p+1)}]^{\top}\right\|_2^2\to\frac{1}{1+\|\balpha\|_2^2}
\end{equation*}
almost surely as $n\to\infty$. Moreover, $\sqrt{n}(\ba_{t}-\balpha)=\Op(1)$ as $n\to\infty$ by~\cite{Pesta2013}. The strong law of large numbers for $\alpha$-mixing by~\cite{chenwu1989} together with Theorem~3.1 by~\cite{P2011} lead to $\ba_{t}-\balpha=o(1)$ almost surely. Since Assumption~\ref{assump:EIVdesign} holds, the expression in~\eqref{eq:swip2} is $\op(1)$. Furthermore, the expression on the right hand side of~\eqref{eq:swip1} is $o(1)$ away from
\begin{equation}\label{eq:approx}
\frac{1}{1+\|\balpha\|_2^2}[\balpha^{\top},-1]\left\{\frac{1}{\sqrt{n}}\left(\bD_{t}-\E\bD_{t}\right)\right\}
\begin{bmatrix}
\balpha\\
-1
\end{bmatrix}
\end{equation}
as $n\to\infty$. Hence, the process from the left hand side of~\eqref{eq:swip1} in $\dist[0,1]$ has approximately the same distribution as the process~\eqref{eq:approx}. 

Note that
\[
[\balpha^{\top},-1]\bD_{t}\begin{bmatrix}
\balpha\\
-1
\end{bmatrix}=(\bar{\Sigma}_{\beps}-\bar{\bSigma}_{\btheta,\beps}^{\top}\balpha)^2\big\|\bY_{[nt]}-\bX_{[nt]}\bbeta\big\|_2^2.
\]
Using the functional central limit theorem for $\alpha$-mixing by~\cite{Herrndorf1983} or~\citet[Corollary 3.2.1]{linlu1997} in an~analogous fashion as in the proof of Theorem~2.3 by~\cite{Pesta2013}, one gets
\[
\left\{[\balpha^{\top},-1]\left\{\frac{1}{\sqrt{n}}\left(\bD_{t}-\E\bD_{t}\right)\right\}\begin{bmatrix}
\balpha\\
-1
\end{bmatrix}\right\}_{t\in[0,1]}\xrightarrow[n\to\infty]{\dist[0,1]}\{\phi^2\upsilon\mathcal{W}(t)\}_{t\in[0,1]}
\]
due to Assumption~\ref{assump:EIVmisfit}.

Similarly for $\left\{\frac{1}{\sqrt{n}}\left(\widetilde{\lambda}_{[n(1-t)]}-[n(1-t)]\sigma^2\right)\right\}_{t\in[0,1]}$ and $\{\widetilde{\mathcal{W}}(t)\}_{t\in[0,1]}$.
\end{proof}

\begin{proof}[Proof of Theorem~\ref{thm:H0}]
The spectral weak invariance principle from Proposition~\ref{prop:SWIP} and Lemma~1 by~\cite{PW2019} in combination with the continuous mapping device complete the proof.
\end{proof}

\begin{proof}[Proof of Theorem~\ref{thm:H1}]
Under~\ref{eq:HA}, let us find a~lower bound for the smallest eigenvalue of the positive semi-definite matrix
\begin{equation}\label{eq:ABCmatrix}
\frac{1}{n}\bSigma^{-1/2}[\bX,\bY]^{\top}[\bX,\bY]\bSigma^{-1/2}=\frac{1}{n}\begin{bmatrix}
\tilde{\bX}^{\top}\tilde{\bX} & \tilde{\bX}^{\top}\tilde{\bY}\\
\tilde{\bY}^{\top}\tilde{\bX} & \tilde{\bY}^{\top}\tilde{\bY}
\end{bmatrix}=:\begin{bmatrix}
\bA & \bc\\
\bc^{\top} & d
\end{bmatrix},
\end{equation}
where $[\tilde{\bX},\tilde{\bY}]:=[\bX,\bY]\bSigma^{-1/2}$. With respect to~\citet[Theorem~1]{D1988}, we get
\begin{equation}\label{eq:lower}
\lambda_{min}\left(\begin{bmatrix}
\bA & \bc\\
\bc^{\top} & d
\end{bmatrix}\right)\geq \frac{d+\ell}{2}-\sqrt{\frac{(d-\ell)^2}{4}+\bc^{\top}\bc},
\end{equation}
where $\ell$ is any lower bound on the smallest eigenvalue of the matrix~$\bA$. Recall that Assumption~\ref{assump:EIVerrors} and the proof of Theorem~3.1 by~\cite{P2011} provide
\begin{equation}\label{eq:limits}
\frac{1}{n}\tilde{\beps}^{\top}\tilde{\beps}\to\sigma^2,\,\frac{1}{n}\tilde{\btheta}^{\top}\tilde{\beps}\to\zero,\,\frac{1}{n}\tilde{\btheta}^{\top}\tilde{\btheta}\to\sigma^2\bI_{p},\,\frac{1}{n}\tilde{\bZ}^{\top}\tilde{\beps}\to 0,\,\frac{1}{n}\tilde{\bZ}^{\top}\tilde{\btheta}\to\zero
\end{equation}
almost surely as $n\to\infty$, where $[\tilde{\btheta},\tilde{\beps}]:=[\btheta,\beps]\bSigma^{-1/2}$ and $\tilde{\bZ}:=\bZ [\bI_{p},\bbeta]\begin{bmatrix}
\bar{\bSigma}_{\btheta}\\
\bar{\bSigma}_{\btheta,\beps}^{\top}
\end{bmatrix}$. By Assumptions~\ref{assump:EIVchange} and~\ref{assump:EIVdesign}, one can obtain
\begin{multline}\label{eq:lower2}
\lambda(\bA)_{min}=\lambda_{min}\left(\frac{1}{n}(\tilde{\bZ}+\tilde{\btheta})^{\top}(\tilde{\bZ}+\tilde{\btheta})\right)\\
\to\lambda_{min}\left([\bar{\bSigma}_{\btheta},\bar{\bSigma}_{\btheta,\beps}]\begin{bmatrix}
\bI_{p}\\
\bbeta^{\top}
\end{bmatrix}\bDelta[\bI_{p},\bbeta]\begin{bmatrix}
\bar{\bSigma}_{\btheta}\\
\bar{\bSigma}_{\btheta,\beps}^{\top}
\end{bmatrix}+\sigma^2\bI_p\right)=\sigma^2+\eta
\end{multline}
almost surely as $n\to\infty$. Relation~\eqref{eq:lower2} immediately provides a~limit of a~candidate for~$\ell$. Now, \eqref{eq:ABCmatrix} and \eqref{eq:lower} lead to
\begin{multline}\label{eq:ineqLimInf}
\liminf_{n\to\infty}\lambda_{min}\left(\frac{1}{n}\bSigma^{-1/2}[\bX,\bY]^{\top}[\bX,\bY]\bSigma^{-1/2}\right)\\
\geq\frac{\displaystyle\lim_{n\to\infty}\frac{1}{n}\tilde{\bY}^{\top}\tilde{\bY}+\sigma^2+\eta}{2}-\sqrt{\frac{\Big(\displaystyle\lim_{n\to\infty}\frac{1}{n}\tilde{\bY}^{\top}\tilde{\bY}-\sigma^2-\eta\Big)^2}{4}+\lim_{n\to\infty}\left\|\frac{1}{n}\tilde{\bX}^{\top}\tilde{\bY}\right\|_2^2}.
\end{multline}
Assumptions~\ref{assump:EIVchange}, \ref{assump:EIVdesign}, and relations~\eqref{eq:limits} yield
\begin{multline*}
\frac{1}{n}\tilde{\bY}^{\top}\tilde{\bY}=\frac{1}{n}\tilde{\bY}_{\tau}^{\top}\tilde{\bY}_{\tau}+\frac{1}{n}\tilde{\bY}_{-\tau}^{\top}\tilde{\bY}_{-\tau}=(\bar{\bSigma}_{\btheta,\beps}^{\top}+\bar{\Sigma}_{\beps}\bbeta^{\top})\bDelta_{\zeta}(\bar{\bSigma}_{\btheta,\beps}+\bbeta\bar{\Sigma}_{\beps})\\
+\sigma^2+(\bar{\bSigma}_{\btheta,\beps}^{\top}+\bar{\Sigma}_{\beps}(\bbeta+\bdelta)^{\top})\bDelta_{-\zeta}(\bar{\bSigma}_{\btheta,\beps}+(\bbeta+\bdelta)\bar{\Sigma}_{\beps})+o(1)=\kappa+\sigma^2+o(1)
\end{multline*}
and
\begin{multline*}
\frac{1}{n}\tilde{\bX}^{\top}\tilde{\bY}=\frac{1}{n}\tilde{\bX}_{\tau}^{\top}\tilde{\bY}_{\tau}+\frac{1}{n}\tilde{\bX}_{-\tau}^{\top}\tilde{\bY}_{-\tau}=(\bar{\bSigma}_{\btheta}+\bar{\Sigma}_{\btheta,\beps}\bbeta^{\top})\bDelta_{\zeta}(\bar{\bSigma}_{\btheta,\beps}+\bbeta\bar{\Sigma}_{\beps})\\
+(\bar{\bSigma}_{\btheta}+\bar{\Sigma}_{\btheta,\beps}(\bbeta+\bdelta)^{\top})\bDelta_{-\zeta}(\bar{\bSigma}_{\btheta,\beps}+(\bbeta+\bdelta)\bar{\Sigma}_{\beps})+o(1)=\bvarphi+o(1)
\end{multline*}
almost surely as $n\to\infty$. Thus, 
\begin{multline}\label{eq:ineqLimInf2}
\frac{\frac{1}{n}\tilde{\bY}^{\top}\tilde{\bY}+\sigma^2+\eta}{2}-\sqrt{\frac{\big(\frac{1}{n}\tilde{\bY}^{\top}\tilde{\bY}-\sigma^2-\eta\big)^2}{4}+\left\|\frac{1}{n}\tilde{\bX}^{\top}\tilde{\bY}\right\|_2^2}\\
=\frac{\kappa+\eta-\sqrt{(\kappa-\eta)^2+4\bvarphi_{2}^{\top}\bvarphi_{2}}}{2}+\sigma^2+o(1)
\end{multline}
almost surely as $n\to\infty$. Hence, combining~\eqref{eq:ineqLimInf} and~\eqref{eq:ineqLimInf2} ends up with
\begin{equation*}
\liminf_{n\to\infty}\lambda_{min}\left(\frac{1}{n}[\tilde{\bX},\tilde{\bY}]^{\top}[\tilde{\bX},\tilde{\bY}]\right)-\sigma^2\geq\lim_{n\to\infty}\frac{2\{\eta\kappa-\bvarphi^{\top}\bvarphi\}}{\kappa+\eta+\sqrt{(\kappa-\eta)^2+4\bvarphi^{\top}\bvarphi}}.
\end{equation*}
Then,
\begin{equation}\label{eq:delta}
\frac{1}{\sqrt{n}}|\lambda_n-n\sigma^2|\xrightarrow[n\to\infty]{\mbox{a.s.}}\infty
\end{equation}
by Assumption~\ref{assump:EIVchange}.

With respect to Assumptions~\ref{assump:EIVdesign}, \ref{assump:EIVerrors}, \ref{assump:EIVmisfit} and according to the underlying proof of Theorem~\ref{thm:H0}, $\frac{1}{\sqrt{n}}\max_{1\leq i<\tau}\big|\lambda_i-\frac{i}{\tau}\lambda_{\tau}\big|$ and $\frac{1}{\sqrt{n}}\max_{\tau< i\leq n}\big|\widetilde{\lambda}_{i}-\frac{n-i}{n-\tau}\widetilde{\lambda}_{\tau}\big|$ are $\Op(1)$ as $n\to\infty$. 
Moreover, $\frac{1}{\sqrt{n}}\big|\lambda_{\tau}-\tau\sigma^2\big|=\Op(1)$ as $n\to\infty$ due to Proposition~\ref{prop:SWIP}.

Note that there are no changes in the linear parameter corresponding to the first~$\tau$ observations as well as to the last (remaining)~$n-\tau$ observations. Let $k=\tau$. Thus, under~\ref{eq:HA},
\begin{align*}
\mathscr{S}_n&\geq\frac{\big|\lambda_{\tau}-\frac{\tau}{n}\lambda_n\big|}{\max_{1\leq i<\tau}\big|\lambda_i-\frac{i}{\tau}\lambda_{\tau}\big|+\max_{\tau< i\leq n}\big|\widetilde{\lambda}_{i}-\frac{n-i}{n-\tau}\widetilde{\lambda}_{\tau}\big|}\\
&\geq\frac{\frac{1}{\sqrt{n}}\Big|\big|\lambda_{\tau}-\tau\sigma^2\big|-\frac{\tau}{n}\big|n\sigma^2-\lambda_n\big|\Big|}{\frac{1}{\sqrt{n}}\max_{1\leq i<\tau}\big|\lambda_i-\frac{i}{\tau}\lambda_{\tau}\big|+\frac{1}{\sqrt{n}}\max_{\tau< i\leq n}\big|\widetilde{\lambda}_{i}-\frac{n-i}{n-\tau}\widetilde{\lambda}_{\tau}\big|}\xrightarrow[n\to\infty]{\prob}\infty,
\end{align*}
because of~\eqref{eq:delta}.

Furthermore, again under~\ref{eq:HA},
\begin{align*}
\mathscr{T}_n&\geq\frac{\big(\lambda_{\tau}-\frac{\tau}{n}\lambda_n\big)^2}{\sum_{i=1}^{\tau-1}\big(\lambda_i-\frac{i}{\tau}\lambda_{\tau}\big)^2+\sum_{i=\tau+1}^{n}\big(\widetilde{\lambda}_{i}-\frac{n-i}{n-\tau}\widetilde{\lambda}_{\tau}\big)^2}\\
&\geq\frac{\frac{1}{n}\Big(\big|\lambda_{\tau}-\tau\sigma^2\big|-\frac{\tau}{n}\big|n\sigma^2-\lambda_n\big|\Big)^2}{\frac{1}{n}\sum_{i=1}^{\tau-1}\big(\lambda_i-\frac{i}{\tau}\lambda_{\tau}\big)^2+\frac{1}{n}\sum_{i=\tau+1}^{n}\big(\widetilde{\lambda}_{i}-\frac{n-i}{n-\tau}\widetilde{\lambda}_{\tau}\big)^2}\xrightarrow[n\to\infty]{\prob}\infty,
\end{align*}
because of similar arguments as in the case of $\mathscr{S}_n$.
\end{proof}

\begin{proof}[Proof of Remark~\ref{rmk:sharp}]
It is sufficient to replace Theorem~1 by~\cite{D1988} with Theorem~3.1 by~\cite{MZ1995} in the proof of Theorem~\ref{thm:H1}.
\end{proof}

\begin{proof}[Proof of Corollary~\ref{cor:est}]
The estimator can be rewritten as
\begin{equation}\label{eq:numdenom}
\hat{\tau}_n=\mathop{\operatorname{argmax}}_{1\leq k \leq n-1}\frac{\frac{1}{n}\big|\lambda_k-\frac{k}{n}\lambda_n\big|+\frac{1}{n}\big|\widetilde{\lambda}_k-\frac{n-k}{n}\widetilde{\lambda}_0\big|}{\max_{1\leq i< k}\frac{1}{\sqrt{n}}\big|\lambda_i-\frac{i}{k}\lambda_k\big|+\max_{k< i\leq n}\frac{1}{\sqrt{n}}\big|\widetilde{\lambda}_{i}-\frac{n-i}{n-k}\widetilde{\lambda}_{k}\big|}.
\end{equation}
We will treat the numerator~$N_n(k)$ and the denominator~$D_n(k)$ of the above stated ratio separately. Let us use notations from the previous proofs and let us recall Assumption~\ref{assump:EIVchange}, \ref{assump:EIVdesign}, and relations~\eqref{eq:limits}. If $[nt]\leq\tau$, then
\begin{equation*}
\frac{1}{n}\tilde{\bY}_{[nt]}^{\top}\tilde{\bY}_{[nt]}=(\bar{\bSigma}_{\btheta,\beps}^{\top}+\bar{\Sigma}_{\beps}\bbeta^{\top})\bDelta_t(\bar{\bSigma}_{\btheta,\beps}+\bbeta\bar{\Sigma}_{\beps})+t\sigma^2+o(1)
\end{equation*}
almost surely as $n\to\infty$. Otherwise, if $[nt]>\tau$, then
\begin{multline*}
\frac{1}{n}\tilde{\bY}_{[nt]}^{\top}\tilde{\bY}_{[nt]}=(\bar{\bSigma}_{\btheta,\beps}^{\top}+\bar{\Sigma}_{\beps}\bbeta^{\top})\bDelta_{\zeta}(\bar{\bSigma}_{\btheta,\beps}+\bbeta\bar{\Sigma}_{\beps})+t\sigma^2\\
+(\bar{\bSigma}_{\btheta,\beps}^{\top}+\bar{\Sigma}_{\beps}(\bbeta+\bdelta)^{\top})(\bDelta_{t}-\bDelta_{\zeta})(\bar{\bSigma}_{\btheta,\beps}+(\bbeta+\bdelta)\bar{\Sigma}_{\beps})+o(1)
\end{multline*}
almost surely as $n\to\infty$. In both cases, we have
\begin{multline*}
\frac{1}{n}[\tilde{\bX}_{[nt]},\tilde{\bY}_{[nt]}]^{\top}[\tilde{\bX}_{[nt]},\tilde{\bY}_{[nt]}]\\
\xrightarrow[n\to\infty]{\mbox{a.s.}}\begin{bmatrix}
\bvartheta^{\top}\bDelta_{t}\bvartheta+t\sigma^2\bI_p & (\bar{\bSigma}_{\btheta}+\bar{\Sigma}_{\btheta,\beps}\bbeta^{\top})\bDelta_{t}(\bar{\bSigma}_{\btheta,\beps}+\bbeta\bar{\Sigma}_{\beps})\\
(\bar{\bSigma}_{\btheta,\beps}^{\top}+\bar{\Sigma}_{\beps}\bbeta^{\top})\bDelta_{t}(\bar{\bSigma}_{\btheta}+\bbeta\bar{\Sigma}_{\btheta,\beps}^{\top}) & (\bar{\bSigma}_{\btheta,\beps}^{\top}+\bar{\Sigma}_{\beps}\bbeta^{\top})\bDelta_t(\bar{\bSigma}_{\btheta,\beps}+\bbeta\bar{\Sigma}_{\beps})+t\sigma^2
\end{bmatrix}\\
=t\sigma^2\bI_{p+1}+\bSigma^{-1/2}\begin{bmatrix}
\bI_p\\
\bbeta^{\top}
\end{bmatrix}\bDelta_{t}[\bI_p,\bbeta]\bSigma^{-1/2}.
\end{multline*}
Therefore, for the Frobenius matrix norm $\|\cdot\|_F$,
\begin{multline*}
\lim_{n\to\infty}\Bigg|\lambda_{min}\left(\frac{1}{n}[\tilde{\bX}_{[nt]},\tilde{\bY}_{[nt]}]^{\top}[\tilde{\bX}_{[nt]},\tilde{\bY}_{[nt]}]\right)-\frac{[nt]}{n}\lambda_{min}\left(\frac{1}{n}[\tilde{\bX},\tilde{\bY}]^{\top}[\tilde{\bX},\tilde{\bY}]\right)\Bigg|\\
=:\lambda_{dif}(t)\leq\left\|\bSigma^{-1/2}\begin{bmatrix}
\bI_p\\
\bbeta^{\top}
\end{bmatrix}\bDelta_{-t}[\bI_p,\bbeta]\bSigma^{-1/2}\right\|_F
\end{multline*}
uniformly in~$t$ almost surely, because $|\lambda_{min}(\bA)-\lambda_{min}(\bB)|\leq\|\bA-\bB\|_F$ due to~\citet[proof of Lemma~2.3]{gallophd}.

For $k=\tau$, Proposition~\ref{prop:SWIP} together with the continuous mapping theorem yield that the denominator from~\eqref{eq:numdenom}
\begin{equation*}
D_n(\tau)\xrightarrow[n\to\infty]{\dist}\frac{\phi^2\upsilon}{1+\|\balpha\|_2^2}\bigg\{\sup_{0\leq t\leq \zeta}\bigg|\mathcal{W}(t)-\frac{t}{\zeta}\mathcal{W}(\zeta)\bigg|+\sup_{\zeta<t\leq 1}\bigg|\widetilde{\mathcal{W}}(t)-\frac{1-t}{1-\zeta}\widetilde{\mathcal{W}}(\zeta)\bigg|\bigg\}=:W,
\end{equation*}
where the limit~$W$ is strictly positive almost surely. We conclude that $|N_n(\tau)/D_n(\tau)|$ converge in distribution to the random variable $\lambda_{dif}(\zeta)+\widetilde{\lambda}_{dif}(\zeta)/W$ such that $\widetilde{\lambda}_{dif}(t):=\lim_{n\to\infty}|\widetilde{\lambda}_{[nt]}-\frac{n-[nt]}{n}\widetilde{\lambda}_{0}|$. For $k=[nt]$ with $t>\zeta$, we obtain
\begin{align*}
&\max_{1\leq i< [nt]}\frac{1}{\sqrt{n}}\bigg|\lambda_i-\frac{i}{[nt]}\lambda_{[nt]}\bigg|+\max_{[nt]< i\leq n}\frac{1}{\sqrt{n}}\bigg|\widetilde{\lambda}_{i}-\frac{n-i}{n-[nt]}\widetilde{\lambda}_{[nt]}\bigg|\\
&\geq\frac{1}{\sqrt{n}}\bigg|\lambda_{[n\zeta]}-\frac{[n\zeta]}{[nt]}\lambda_{[nt]}\bigg|\geq\frac{1}{\sqrt{n}}\Bigg|\bigg|\lambda_{[n\zeta]}-[n\zeta]\sigma^2\bigg|-\frac{[n\zeta]}{[nt]}\bigg|\lambda_{[nt]}-[nt]\sigma^2\bigg|\Bigg|\\
&\approx\Bigg|\Op(1)-\sqrt{n}\frac{\zeta}{t}\bigg|\frac{\lambda_{[nt]}}{n}-t\sigma^2\bigg|\Bigg|\\
&\approx\Bigg|\Op(1)-2\sqrt{n}\frac{\zeta}{t}\frac{\big|\eta(t)\kappa(t)-\bvarphi(t)^{\top}\bvarphi(t)\big|}{\kappa(t)+\eta(t)+\sqrt{(\kappa(t)-\eta(t))^2+4\bvarphi(t)^{\top}\bvarphi(t)}}\Bigg|\xrightarrow[n\to\infty]{\prob}\infty
\end{align*}
according to the proof of Theorem~\ref{thm:H1} and assumption~\eqref{eq:Cpt}. Similar arguments can be applied in the case~$t<\zeta$ and the convergence holds uniformly for all~$t$ outside any $\epsilon$-neighborhood of~$\zeta$. It follows that for an arbitrary $\epsilon>0$,
\[
\max_{k: |k-\tau|\geq n\epsilon}\frac{|N_n(k)|}{D_n(k)}=\Op\left(\frac{1}{|\eta(t)\kappa(t)-\bvarphi(t)^{\top}\bvarphi(t)|\sqrt{n}}\right).
\]
Now, let us chose a~sequence $d_n\to 0$ with $d_n|\eta(t)\kappa(t)-\bvarphi(t)^{\top}\bvarphi(t)|\sqrt{n}\to\infty$. Then, for any $\epsilon>0$,
\begin{equation*}
\P[|\hat{\tau}/n-\zeta|>\epsilon]\leq \P[|N_n(\tau)/D_n(\tau)|<d_n]+\P\bigg[\max_{k: |k-\tau|\geq n\epsilon}|N_n(k)/D_n(k)|>d_n\bigg]\xrightarrow{n\to\infty}0.
\end{equation*}
\end{proof}

\subsection*{Acknowledgements}
The research of Michal Pe\v{s}ta was supported by the Czech Science Foundation project GA\v{C}R No.~18-01781Y.

\begingroup
\setlength{\bibsep}{4pt}
\bibliography{Pesta-CPLR}
\endgroup

\end{document}